\documentclass[10pt]{article}

\title{Arbitrarily high-order (weighted) essentially non-oscillatory finite difference schemes for anelastic flows on staggered meshes}
\author{Siddhartha Mishra\thanks{Seminar for Applied Mathematics, ETH Z\"urich, R\"amistrasse 101, Z\"urich, Switzerland.}, Carlos Par\'{e}s-Pulido\thanks{Seminar for Applied Mathematics, ETH Z\"urich, R\"amistrasse 101, Z\"urich, Switzerland.}  and Kyle G. Pressel\thanks{Pacific Northwest National Laboratory, P.O. Box 999, Richland, WA 99352, USA}} 
\usepackage{tikz}

\usepackage{fullpage} 
\usepackage{lineno}
\usepackage[margin=1.5cm]{caption}
\usepackage{subcaption}
\usepackage{amsmath}
\usepackage{bm}
\usepackage{amsthm}
\usepackage{amsfonts}
\usepackage[toc,page]{appendix}
\usepackage{color,soul}

\newtheorem{theorem}{Theorem}
\newtheorem{lemma}{Lemma}

\newtheoremstyle{algorithm}
{}
{}
{}
{}
{\bfseries}
{.}
{\newline}
{}

\theoremstyle{algorithm}
\newtheorem{algorithm}{Algorithm}

\theoremstyle{definition}
\newtheorem{remark}{Remark}
\newtheorem{proposition}{Proposition}
\graphicspath{{./Figures/}}

\numberwithin{equation}{section}

\newcommand{\rev}[1]{#1}
\newcommand{\revb}[1]{#1}

\begin{document}
	
	\maketitle
	\begin{abstract}
		We propose a WENO finite difference scheme to approximate anelastic flows, and scalars advected by them, on staggered grids. In contrast to existing WENO schemes on staggered grids, the proposed scheme is designed to be arbitrarily high-order accurate as it judiciously combines ENO interpolations of velocities with WENO reconstructions of spatial derivatives. A set of numerical experiments are presented to demonstrate the increase in accuracy and robustness with the proposed scheme, when compared to existing WENO schemes and state-of-the-art central finite difference schemes. 				
		
	\end{abstract}
	\section{Introduction}
	
	In numerical modeling of fluid systems that are characterized by a low Mach number, it is often advantageous to introduce approximations to the fully compressible equations that eliminate acoustic waves. Doing so allows the explicit time integration of the governing equations to take much longer time steps in comparison to a similar integration of the compressible equations. This is because numerical stability in the soundproofed system does not depend on the phase velocity of acoustic waves. Various approximations to the compressible equations that eliminate acoustic waves have been developed including the incompressible, Boussinesq, and pseudo-incompressible approximations. The choice of which approximation is used is determined by the properties of the system being studied. 
	
	In atmospheric science, the anelastic system of equations is widely used as the basis for limited area models, because many atmospheric flows are buoyancy driven, and the stratification of the atmosphere makes it natural to assume that vertical gradients of density are much larger than horizontal gradients. In the anelastic system, the momentum, entropy, and continuity equations are given by
	\begin{align}
		\frac{\partial U}{\partial t} + \frac{1}{\rho_0} \nabla \cdot (\rho_0 U \otimes U) &= - \nabla \left(\frac{p'}{\rho_0} \right) + b e_3,  \label{eq:momentum} \\
		\frac{\partial s}{\partial t} + \frac{1}{\rho_0} \nabla \cdot \left(\rho_0 U s \right) = 0 \label{eq:entropy} \\
		\nabla \cdot (\rho_0 U) & = 0 \label{eq:continuity}
	\end{align}
	respectively.  Here $U = \left(u_1, u_2, u_3 \right) \in \mathbb{R}^3$ is the fluid velocity, $\rho_0$ is a horizontally homogeneous reference state density, $p^\prime = p - p_0$ is the dynamic pressure, $b$ is the buoyancy, $e_3 = (0,0,1)$, and $s$ is the specific entropy defined as in \cite{Pressel2015}. Following \cite{Pauluis2008}, we define $\rho_0$ to be consistent with an isentropic and hydrostatic state and where the buoyancy $b = g \left(\rho/\rho_0 - 1 \right)$ couples the momentum to thermodynamics, and $\rho$ is determined from the equation of state. For most of this paper, we focus on equations \eqref{eq:momentum} and \eqref{eq:continuity}, as with this choice, eq. \eqref{eq:entropy} is only coupled to the other two only through $b = b(\rho(s))$, and can be treated as a source term.
	
	Given their importance in applications, a large variety of numerical methods for approximating anelastic (incompressible) flows are available. In applications where complex domain geometries are rarely encountered, finite difference schemes are a popular discretization framework as they can account for more general boundary conditions than spectral methods \cite{Gottlieb1984}. \rev{Alternative numerical frameworks, such as the finite element method, discontinuous Galerkin methods, etc. are possible approaches for these problems, particularly on domains with complex geometries. However, on account of their computational efficiency and simplicity of implementation, finite difference methods remain a very attractive option for many applications, particularly in climate and weather modelling.}
	
	The most straightforward finite difference discretizations of \eqref{eq:momentum} and \eqref{eq:continuity} are on \emph{collocated} meshes, where one evolves point values of the velocity and pressure at cell centers. However, it is well known that this procedure can lead to what is termed as \emph{velocity-pressure decoupling}, \cite{Patankar1980} and references therein. Consequently, nonphysical checkerboard modes are obtained as solutions to the (discrete) elliptic equation that one has to solve in order to compute the pressure from the continuity equation \eqref{eq:continuity}.
	
	A possible remedy for these nonphysical numerical artifacts is the use of \emph{staggered meshes}. In this framework, each velocity component is discretized on the center of the underlying normal cell edge. The pressure (and any passively or actively advected scalars) is discretized at cell centers. The spatial derivatives in \eqref{eq:momentum} can then be discretized by (high-order) central differences. A large number of such \emph{central schemes} exist and are employed in many atmospheric codes. 
	
	In many applications, the Reynolds number of the underlying flow is very high and description of the dynamics requires computing contributions across a large range of scales. Given the infeasibility of direct numerical simulation (DNS) in such problems, one resorts to an additional turbulence model for \eqref{eq:momentum}. \emph{Large eddy simulations} (LES) are a very attractive framework for turbulence modeling. In LES, a certain range of scales are numerically resolved and remaining \emph{sub-grid scales} (SGS) are modeled by a suitable closure model; a detailed discussion can be found in \cite{Pope2000, Sagaut2006}.
	Many closure models are based on adding eddy diffusivity in terms of grid-scale gradients (\cite{Smagorinksy1963, Lilly1966}, further references in \cite{Pope2004}).
	
	As the leading order truncation terms for central finite difference schemes are dispersive in nature, the dominant source of numerical dissipation is delegated to the SGS terms. Hence, central finite difference schemes are widely used in the context of LES for anelastic flows. However, in many recent papers, it has been observed that central difference schemes together with SGS terms may not provide enough numerical diffusion to enable a stable computation. For instance, in \cite{Pressel2017}, the authors demonstrated that the lack of numerical dissipation and the generation of dispersive oscillations by central finite difference schemes led to very inaccurate computation of stratocumulus clouds. In particular, there was not enough dissipation in the mixing zone, where sharp gradients (at grid scale) are encountered which lead to excessive SGS scalar fluxes. Given this well-known inadequacy of central finite difference schemes in resolving sharp grid scale gradients, these schemes can be very inaccurate in the LES of a significant class of turbulent flows. Moreover, a crucial constraint of any numerical simulation of anelastic flows is to retain positivity and bound preservation for advected scalars. It is well-known that central difference schemes can yield spurious oscillations that may lead to a breakdown of positivity of advected scalars.
	
	An alternative discretization framework for such flows can be provided by (\emph{Weighted}-) \emph{essentially non-oscillatory} or WENO schemes. These schemes were first developed in \cite{Harten1987,Shu1989,Shu1997} in the context of robust approximation of compressible flows which are characterized by discontinuities such as shock waves. In this discretization framework, either cell averages or point values are reconstructed from a piecewise polynomial interpolation, based on a clever stencil selection or stencil weighing procedure that tracks the direction of smoothness of the underlying functions. Consequently, a precise amount of numerical diffusion is added in the vicinity of sharp gradients (at grid scale) to stabilize computations. Moreover, the WENO procedure yields schemes with arbitrarily high order of accuracy. 
	
	WENO schemes for anelastic flows on collocated grids were developed in \cite{Balsara2000} and references therein. WENO schemes for anelastic flows on staggered grids have been described recently in \cite{Pressel2015,Pressel2017} and references therein. In these articles, the authors adapted the WENO procedure to approximate anelastic flows on staggered grids and used these schemes for LES of stratocumulus clouds. The results were very satisfactory and a significant gain in accuracy (with respect to observations) was obtained over central finite difference schemes. However a detailed analysis of the WENO schemes of \cite{Pressel2015,Pressel2017} reveals that although arbitrarily high-order WENO interpolations are employed, the full scheme is at most \emph{second-order} accurate (see section \ref{sec:2} for details). 
	
	As a very high order of accuracy is beneficial for resolving flows with significant range of small scales such as in LES \cite{Ghosal1996}, our aim in this article is to design WENO schemes with these properties. In particular, we design finite difference schemes that are
	\begin{itemize}
		\item Able to resolve sharp gradients without oscillations. 
		
		\item (Formally) arbitrarily high-order accurate for the momentum equation \eqref{eq:momentum} of anelastic flows.
		
		\item Formally arbitrarily high-order accurate for scalars, passively or actively advected by the flow.

	\end{itemize}
	We design these schemes by judiciously combining a WENO procedure with an additional essentially non-oscillatory (ENO) interpolation step to provide point values of fields. The arbitrarily high-order scheme for scalars is based on a non-conservative formulation. 
	
	The rest of the paper is organized as follows. In section \ref{sec:2}, we describe the existing WENO schemes of \cite{Pressel2015,Pressel2017} on a staggered grid and prove that they are at most second-order accurate. The novel WENO scheme of this paper is presented in section \ref{sec:3} and numerical experiments are presented in section \ref{sec:4}. In section \ref{sec:5}, we present the arbitrarily high-order schemes for scalars, advected by the anelastic flow.
	\section{Existing WENO schemes on staggered grids.}
	\label{sec:2}
	For simplicity of notation and exposition, we consider here the two-dimensional version of the momentum equation \eqref{eq:momentum}. The corresponding velocity field is denoted as $U = (u_1,u_2)$; for convenience of notation, we will also use $U = (u, v)$.
	\subsection{The Arakawa C-grid}
	The computational domain is the rectangle $(x_{min}, x_{max}) \times (y_{min}, y_{max}) \subset \mathbb{R}^2$. This rectangle is partitioned into a grid of $N_x \times N_y$ cells $C_{ij} = [x_i, x_{i+1}) \times [y_j, y_{j+1})$, $\forall(i,j) \in \{0,1,\dots,N_x\} \times \{0,1,\dots,N_y\}$, with $x_i = x_{min} + i\Delta x,\  y_j = y_{min} + j \Delta y$ and  $\Delta x = \frac{x_{max}-x_{min}}{N_x},\ \Delta y = \frac{y_{max}-y_{min}}{N_y}$. The grid cell centers are $(x_{i+\frac{1}{2}}, y_{j + \frac{1}{2}})$ with $ x_{i + \frac{1}{2}} = {x_{min}} + (i + \frac{1}{2}) \Delta x,\ y_{j + \frac{1}{2}} = y_{min} + (j + \frac{1}{2}) \Delta y$.
	
	On this grid, we consider what is called an Arakawa-C \cite{Arakawa1977} staggering of variables on the computational grid, depicted in Figure \ref{fig:staggeredgrid}. 
	\rev{Alternative choices of spatial arrangement of values, such as Arakawa-A (collocated), or Arakawa-B, with all components of velocity evaluated at the same location, may produce unstable checkerboard modes and are unsuitable for projection-based solvers for incompressible flows.}
	In this configuration, scalars (see section \ref{sec:5}) lie at grid cell centers (denoted $\phi_{i + \frac{1}{2}, j + \frac{1}{2}} := \phi(x_{i+\frac{1}{2}}, y_{j+\frac{1}{2}})$); pressure is diagnostically computed at those points as well. Conversely, velocities lie at the center of the cell interfaces normal to their direction; specifically, $u$ velocities are located at $(x_i, y_{j + \frac{1}{2}})$, and $v$ at $(x_{i + \frac{1}{2}}, y_j)$. We denote these as $u_{i, j+\frac{1}{2}}  := u(x_{i}, y_{j+\frac{1}{2}})$, $v_{i + \frac{1}{2}, j} := v(x_{i+\frac{1}{2}}, y_{j})$. 
	
	\begin{figure}
		\centering
		\begin{tikzpicture}
		
		\draw[step=2cm,gray,very thin] (-0.5, -0.5) grid (6.5, 6.5);
		\draw[thick, line width=0.7mm] (2,2) rectangle (4,4);
		\node at (-0.5,0.2) {$y_{j-1}$};
		\node at (-0.5,2.2) {$y_{j}$};
		\node at (-0.5,4.2) {$y_{j+1}$};
		\node at (-0.5,6.2) {$y_{j+2}$};
		\node at (0.5, -0.3) {$x_{i-1}$};
		\node at (2.3, -0.3) {$x_{i}$};
		\node at (4.5, -0.3) {$x_{i+1}$};
		\node at (6.5, -0.3) {$x_{i+2}$};
		
		\fill (3,3) circle[radius=2pt];
		\node[scale=0.75] at (3,3.3) {$\{p, \phi\}_{i + \frac{1}{2}, j + \frac{1}{2}}$};
		\draw[->] (3, 1.8) -- (3, 2.2);
		\node[scale=0.75] at (3.5,2.3) {$v_{i + \frac{1}{2}, j}$};
		\draw[->] (1.8,3) -- (2.2,3);
		\node[scale=0.75] at (2.5, 2.7) {$u_{i, j + \frac{1}{2}}$};
		
		\fill (3,5) circle[radius=2pt];
		\node[scale=0.75] at (3,5.3) {$\{p, \phi\}_{i + \frac{1}{2}, j + \frac{3}{2}}$};
		\draw[->] (3, 3.8) -- (3, 4.2);
		\node[scale=0.75] at (3.5,4.3) {$v_{i + \frac{1}{2}, j+1}$};
		\draw[->] (1.8,5) -- (2.2,5);
		\node[scale=0.75] at (2.5, 4.7) {$u_{i, j + \frac{3}{2}}$};
		
		\fill (5,3) circle[radius=2pt];
		\node[scale=0.75] at (5,3.3) {$\{p, \phi\}_{i + \frac{3}{2}, j + \frac{1}{2}}$};
		\draw[->] (5, 1.8) -- (5, 2.2);
		\node[scale=0.75] at (5.5,2.3) {$v_{i + \frac{3}{2}, j}$};
		\draw[->] (3.8,3) -- (4.2,3);
		\node[scale=0.75] at (4.7, 2.7) {$u_{i+1, j + \frac{1}{2}}$};
		
		\fill (3,1) circle[radius=2pt];
		\node[scale=0.75] at (3,1.3) {$\{p, \phi\}_{i + \frac{1}{2}, j - \frac{1}{2}}$};
		\draw[->] (3, -0.2) -- (3, 0.2);
		\node[scale=0.75] at (3.5,0.3) {$v_{i + \frac{1}{2}, j-1}$};
		\draw[->] (1.8,1) -- (2.2,1);
		\node[scale=0.75] at (2.5, 0.7) {$u_{i, j - \frac{1}{2}}$};
		
		\fill (1,3) circle[radius=2pt];
		\node[scale=0.75] at (1,3.3) {$\{p, \phi\}_{i - \frac{1}{2}, j + \frac{1}{2}}$};
		\draw[->] (1, 1.8) -- (1, 2.2);
		\node[scale=0.75] at (1.5,2.3) {$v_{i - \frac{1}{2}, j}$};
		\draw[->] (-0.2,3) -- (0.2,3);
		\node[scale=0.75] at (0.6, 2.7) {$u_{i-1, j + \frac{1}{2}}$};
		
		\fill (5,5) circle[radius=2pt];
		\draw[->] (5, 3.8) -- (5, 4.2);
		\draw[->] (3.8,5) -- (4.2,5);
		
		\draw[->] (5.8,3) -- (6.2,3);
		\draw[->] (5.8,5) -- (6.2,5);
		\draw[->] (5.8,1) -- (6.2,1);
		
		\draw[->] (3, 5.8) -- (3, 6.2);
		\draw[->] (5, 5.8) -- (5, 6.2);
		\draw[->] (1, 5.8) -- (1, 6.2);
		
		\fill (1,5) circle[radius=2pt];
		\draw[->] (1, 3.8) -- (1, 4.2);
		\draw[->] (-0.2,5) -- (0.2,5);
		
		\fill (1,1) circle[radius=2pt];
		\draw[->] (1, -0.2) -- (1, 0.2);
		\draw[->] (-0.2,1) -- (0.2,1);
		
		\fill (5,1) circle[radius=2pt];
		\draw[->] (5, -0.2) -- (5, 0.2);
		\draw[->] (3.8,1) -- (4.2,1);
		
		\end{tikzpicture}
		\caption{The staggered Arakawa-C grid in 2D. Velocities $(u,v)$ are stored at cell interfaces, scalars $\phi$ at cell centers. Pressure $p'$ is computed at cell centers as well.}
		\label{fig:staggeredgrid}
	\end{figure}
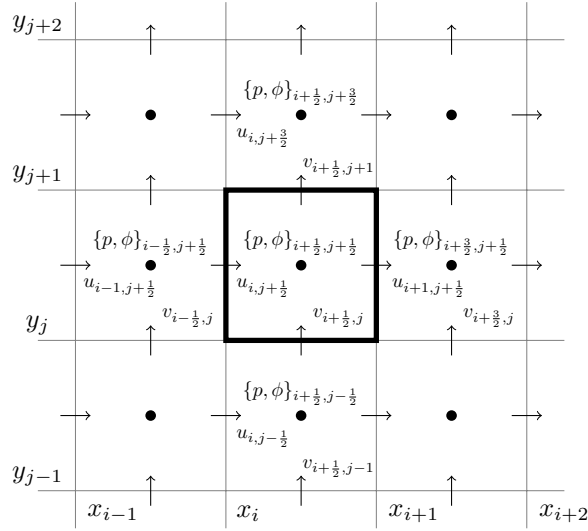
	Our aim is to design a WENO scheme on this grid to approximate the anelastic equations of motion. We start by providing a very brief description of the WENO framework for the sake of completeness.

	\subsection{An overview of the ENO and WENO methodology}
	Given a grid and the cell averages (resp. point values) of a function on the grid, ENO and WENO methods provide a reconstruction (resp. interpolation) of the function at a set of points on the grid in terms of piecewise polynomials. See \cite{Harten1987, Shu1997, Shu1989} and references therein.
	
	For simplicity, we consider a function $w:[x_L,x_R] \rightarrow {\mathbb R}$ on a 1D grid with cells $C_j := [x_{j-\frac{1}{2}}, x_{j+\frac{1}{2}})$ and aim to reconstruct piecewise polynomials from cell averages, with formal order of accuracy $k \in \mathbb{N}$. To this end, the key idea behind ENO is a careful choice of the stencil; i.e. the set of $k$ cells $S^{(r)}(j) = \{ C_{j-r}, \dots, C_{j-r+k-1}\}$, $r  \in \{0, \dots, k-1 \}$, from which information is taken. This choice is based on the smoothness of the function: out of the available stencils, the interpolation corresponds to one where the function is the \emph{most regular}. This measurement of smoothness can be performed with Newton's divided differences. Denoting the chosen stencil by $r$, and denoting by $\bar{w}_j$ the cell average of $w$ in cell $C_j$, the reconstruction has the form:
	
	\[ w_{j+\frac{1}{2}}^{(r)} := \sum_{i=0}^{k-1} c_{ri} \bar{w}_{i-r+j} = w(x_{j+\frac{1}{2}}^-) + O(\Delta x^k), \]
	where the constants $c_{ri}$ are reconstruction weights, depending on stencil $r$, order $k$ of the interpolation, and on the grid, but not on the underlying function $w$.
	
	Even if information from $2k-1$ cells is used in the ENO method, the order of accuracy is limited to $k$. This redundancy is alleviated in the WENO method, in which the reconstruction is based on a convex combination of all candidate ENO reconstructions. The weights in this combination are very small for approximations coming from stencils with discontinuities, and we can still guarantee order of accuracy $2k-1$ when the function is sufficiently smooth in all stencils. To be more specific, the reconstruction at the interface of cell $C_j$ has the form \[ w_{j+\frac{1}{2}} := \sum_{r=0}^{k-1} \omega_r w^{(r)}_{j+\frac{1}{2}} = w(x_{j+\frac{1}{2}}^-) + O(\Delta x^{2k-1}), \]
	where $\omega_r \propto \frac{d_r}{(\epsilon + \beta_r)^p}$ is normalized so that $\sum_r \omega_r = 1$. Weights $d_r$ are chosen to produce high-order reconstructions when $w$ is smooth, and $\beta_r$ is a smoothness indicator, such that $\beta_r \approx 0$ if $v$ is smooth in $S^{(r)}(j)$, and $\beta_r = O(1)$ otherwise. Following \cite{Shu1997}, we set $\epsilon$ to be $10^{-10}$ for our experiments, and $p=2$.
	
	As we work with a finite difference framework, a priori it is unclear if the reconstruction procedure outlined above can be applied in the current context. However, we are only interested in a high-order approximation of the spatial derivatives in \eqref{eq:momentum}, and not in a very accurate approximation of point values. Therefore, we can use the novel trick of Shu \cite{Shu1997} by which he converts reconstruction of cell averages into approximation of first derivatives. For the sake of completeness, we present the following proposition,
	\begin{proposition} \label{remark:wenoDeriv} 
		Given the point values of a smooth function $f$ on the cell centers of a uniform one-dimensional grid with step $\Delta x$, we can perform a WENO reconstruction to obtain \emph{numerical fluxes} $\hat{f}$ such that the spatial derivative of $f$ is approximated to a high-order of accuracy, i.e, 
		\begin{linenomath*}\begin{equation}
				\label{eq:wenoDiff} \frac{1}{\Delta x}(\hat{f}_{i + \frac{1}{2}} - \hat{f}_{i - \frac{1}{2}}) = f'(x_i) + O(\Delta x ^{k-1})
		\end{equation} \end{linenomath*}
	\end{proposition}
	
	\begin{proof}
		We follow Shu \cite{Shu1997} and seek a function $h$, depending on $\Delta x$, implicitly defined such that its sliding cell average is $f$. That is,
		\begin{equation}
			\label{eq:hdef}
			f(x) = \frac{1}{\Delta x} \int_{x - \frac{\Delta x}{2}}^{x + \frac{\Delta x}{2}} h(s) ds
		\end{equation} 
		A straightforward application of the fundamental theorem of calculus yields,
		
		\begin{equation}\label{eq:ftc} f'(x) = \frac{1}{\Delta x} \left( h\left(x + \frac{\Delta x}{2}\right) - h\left(x - \frac{\Delta x}{2}\right) \right) \end{equation}
		
		On the other hand, the cell averages of $h$ verify that:
		\[ \bar{h}_i := \frac{1}{\Delta x} \int_{x_i - \frac{\Delta x}{2}}^{x_i + \frac{\Delta x}{2}} h(s) ds \stackrel{\eqref{eq:hdef}}{=} f(x_i) \]
		Now, we can apply standard WENO reconstruction, as outlined above, to the cell averages $\bar{h}_i$ to obtain 
		\begin{equation} \hat{h}_{i+\frac{1}{2}} = h\left(x_i + \frac{\Delta x}{2}\right) + O(\Delta x^k) \label{eq:deriverrorterm} \end{equation}
		And finally, inserting in \eqref{eq:ftc},
		\[ f'(x_i) = \frac{1}{\Delta x} \left( \hat{h}_{i+\frac{1}{2}} - \hat{h}_{i-\frac{1}{2}} \right) + O(\Delta x^{k-1}) \]
		
		So it suffices to let $\hat{f}_{i+\frac{1}{2}} := \hat{h}_{i+\frac{1}{2}}$ and conclude the proof.

	\end{proof}
	\begin{remark}
	\rev{
	The accuracy of order $k-1$ in eq. \eqref{eq:wenoDiff} is a sharp bound. In practice however, one usually obtains a better, $k$-th order approximation to the spatial derivative, as numerical examples in section \ref{sec:4} show. If the error term $O(\Delta x^k)$ in eq. \eqref{eq:deriverrorterm} is smooth, its difference in eq. \eqref{eq:ftc} will be itself in $O(\Delta x^{k+1})$, producing an approximation to the derivative of order $k$. However, the smoothness of the error term requires the underlying function $h$ to be $k+1$ times continuously differentiable. In case $h$ is only $k$-times continuously differentiable, as required for the Taylor expansion in \eqref{eq:deriverrorterm}, then the approximation will only be of order $k-1$. Therefore, we will follow the slight abuse of notation, well established in the literature, of referring to these schemes as $(2k-1)$st-order accurate WENO schemes, even if for some cases, they only approximate the derivative to order $2k-2$.
	}
	\end{remark}
	\subsubsection{Further notation for WENO reconstructions.}
	Let $f: \mathbb{R}\to\mathbb{R}$ be a function, and let \[ \bar{f}_i := \frac{1}{\Delta x} \int_{x_{i-\frac{1}{2}}}^{x_{i+\frac{1}{2}}} f(x)dx,\] for $i\in I$ be its cell averages. Fix $k \in \mathbb{N}$, we term $W_{-\frac{1}{2}}, W_{+\frac{1}{2}}: \mathbb{R}^{2k-1} \to \mathbb{R}$ as the WENO-$(2k-1)$ reconstruction operators that provide approximations to the point values at the left and right cell edge respectively, which if $f$ is sufficiently regular are high-order accurate; i.e.:
	\[ W_{+\frac{1}{2}}( \bar{f}_{i-(k-1)}, \bar{f}_{i-(k-2)}, \dots, \bar{f}_i, \dots, \bar{f}_{i+(k-1)} ) = f(x_{i+\frac{1}{2}}^-) + O(\Delta x^{2k-1}) \]
	\[ W_{-\frac{1}{2}}( \bar{f}_{i-(k-1)}, \bar{f}_{i-(k-2)}, \dots, \bar{f}_i, \dots, \bar{f}_{i+(k-1)} ) = f(x_{i-\frac{1}{2}}^+) + O(\Delta x^{2k-1}) \]

	\subsection{The WENO scheme used by Pressel et al. \cite{Pressel2015,Pressel2017}}
	
	We start by describing our discretization of the advection terms in the momentum equation \eqref{eq:momentum}. To this end we assume uniform reference density and dynamic pressure $\rho_0,\,p^{\prime} \equiv 1$ and buoyancy $b\equiv 0$ in \eqref{eq:momentum} and write it
	down in two space dimensions as
	\begin{equation}
		\begin{aligned}
			\label{eq:simple_mom}
			\frac{\partial u}{\partial t} &= - \frac{\partial}{\partial x}(u^2) - \frac{\partial}{\partial y}(uv) , \\
			\frac{\partial v}{\partial t} &= - \frac{\partial}{\partial x}(uv) - \frac{\partial}{\partial y}(v^2).
		\end{aligned}
	\end{equation} 
	\rev{In particular, we will focus on a conservative WENO discretization of the first equation in \eqref{eq:simple_mom}. The choice of a conservative scheme is very natural: both physically, due to conservation of momentum, and mathematically, as for non-smooth flows $(u,v)$, terms such as $u v_y$ may not be well-defined, whereas the derivatives in this formulation are well-defined in a weak sense. Moreover, even if the flows are smooth but with large gradients, then, at the resolved scales, using non-conservative products might lead to larger numerical errors than the conservative form.} Following \cite{Pressel2015}, a semi-discrete WENO scheme for it can be written as 
	\begin{linenomath*}\begin{equation}\label{eq:discrete_simple_mom}
			\frac{\partial}{\partial t}u(x_i,y_{j+\frac{1}{2}},t) = -\frac{1}{\Delta x}\left(F^{u,x}_{i+\frac{1}{2}, j+\frac{1}{2}}(t) - F^{u,x}_{i - \frac{1}{2}, j+\frac{1}{2}}(t)\right)  - \frac{1}{\Delta y}\left(F^{u,y}_{i,j+1}(t) - F^{u,y}_{i,j}(t)\right).
	\end{equation}\end{linenomath*}
	Flux reconstructions $F$ (for the sake of clarity, we omit the time dependence throughout) are given by a generic reconstruction operator $\widehat{(\cdot)}$, to be described in the sequel:
	\begin{linenomath*}\begin{align*}
			F^{u,x}_{i+\frac{1}{2},j+\frac{1}{2}} &= \widehat{u^2}_{i+\frac{1}{2}, j+\frac{1}{2}}, \quad 
			F^{u,y}_{i,j} = \widehat{uv}_{i, j} .
	\end{align*}\end{linenomath*}
	The reconstruction $F_{i+\frac{1}{2},j+\frac{1}{2}}^{u,x}$ corresponds to the function $u^2$ along the $x$-direction. It can be computed directly by performing a WENO reconstruction from the point values
	$u^2_{i,j+\frac{1}{2}}$ in the $x$-direction.
	
	On the other hand, the reconstruction $F_{i,j}^{u,y}$ corresponds to approximating a $y$-spatial derivative of the function $uv$. This is problematic as the values of the velocities $u$ and $v$ are \emph{staggered} and
	are not defined at the same locations. In particular, $u$ is defined at points $(x_{i},y_{j+\frac{1}{2}})$ whereas $v$ is defined at points $(x_{i+\frac{1}{2}}, y_j)$. This situation is depicted in Figure \ref{fig:flux_recon}. Thus, we have to perform an interpolation procedure to be able to obtain point values for the
	function $uv$, which can then provide the input to a WENO reconstruction procedure for calculating $F^{u,y}_{i,j}$ in \eqref{eq:discrete_simple_mom}.
	
	\begin{figure}
		\begin{center}
			\begin{tikzpicture}
			
			\draw[step=2cm,gray,very thin] (-2.5,-1.7) grid (2.5,3.5);
			\draw[thick, line width=0.7mm] (0,0) rectangle (2,2);
			\node at (-2.2,0.2) {$y_{j}$};
			\node at (-2.4,2.2) {$y_{j+1}$};
			\node at (-2.4, -0.8) {$x_{i-1}$};
			\node at (-0.2, -0.8) {$x_{i}$};
			
			\draw[->] (-2.2, 1) -- (-1.8, 1);
			\draw[->] (-0.2, 1) -- (0.2, 1);
			\draw[->] (1.8, 1) -- (2.2, 1);
			\draw[->] (-1, -0.2) -- (-1, 0.2);
			\draw[->] (1, -0.2) -- (1, 0.2);
			\draw[->] (-1, 1.8) -- (-1, 2.2);
			\draw[->] (1, 1.8) -- (1, 2.2);
			
			\draw[->] (-0.2, 3) -- (0.2, 3);
			\draw[->] (-0.2, -1) -- (0.2, -1);
			
			\draw (0,2) circle[radius = 5pt];
			\draw (0,0) circle[radius = 5pt];
			\node at (0.6, 1.6) {$F^{u,y}_{i, j+1}$};
			\node at (0.4, -0.4) {$F^{u,y}_{i, j}$};
			\node at (0.6, 0.7) {$u_{i, j+\frac{1}{2}}$};
			\node at (1.5, 0.3) {$v_{i+\frac{1}{2}, j}$};
			\node at (-1, 2.4) {$v_{i-\frac{1}{2}, j+1}$};
			\node at (1.2, 2.4) {$v_{i+\frac{1}{2}, j+1}$};
			
			\end{tikzpicture}
			\caption{Reconstruction of $\frac{\partial}{\partial y}(uv)$ on a staggered grid. Circles denote points where fluxes need to be computed.} 
			\label{fig:flux_recon}
		\end{center}
	\end{figure}
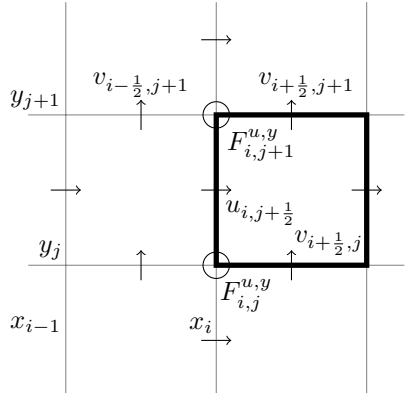
	One possible approach to reconstructing the $uv$ momentum flux, that has been used in applications (e.g. Pressel et al. \cite{Pressel2015}\footnote{In \cite{Pressel2015}, this idea of interpolating the advecting velocity is also used for the reconstruction of the $(u_i)^2$ fluxes}, Skamarock et al. \cite{Skamarock2008WRF}), is to use a high-order one-dimensional central interpolation to reconstruct the advecting velocity $v$, and WENO interpolation to reconstruct the velocity $u$. This procedure can be summarized as:
	
	\begin{algorithm}
		\label{algo:pycles}
		~\textbf{Goal}: Given $k \in \mathbb{N}$, find a reconstruction of the numerical flux, $F^{u,y}_{i,j+1}$, such that, formally, $\forall i,j$: 
		\begin{equation}\label{eq:unmetgoal} \frac{1}{\Delta y}({F^{u,y}_{i, j+1} - F^{u,y}_{i,j}}) = \frac{\partial}{\partial y} (uv) (x_{i}, y_{j+\frac{1}{2}}) + O(\Delta x^{2k-1} + \Delta y^{2k-1}) \end{equation}
		\begin{enumerate}
			\item Using 1D central interpolation 
			on the horizontal, symmetric, $2k$-point stencil
			$\{ v_{i-\frac{1}{2} + l,~j+1} \}_{l=-k+1}^{k}$, compute
			\begin{linenomath*}\begin{equation*} 
					\tilde{v}_{i,j+1} = v(x_i, y_{j+1}) + O(\Delta x^{2k})
			\end{equation*}\end{linenomath*}
			
			\item Compute the two possible biased, vertical $(2k-1)$-WENO reconstructions of $u$ at $(x_i, y_{j+1})$:
			\begin{align*} \hat{u}_{i,j+1}^- &:= W_{+\frac{1}{2}}(u_{i, j+\frac{1}{2}+(-k+1)}, u_{i, j+\frac{1}{2}+(-k+2)}, \dots, u_{i, j+\frac{1}{2}+k-1}) \\
				\hat{u}_{i,j+1}^+ &:= W_{-\frac{1}{2}}(u_{i, j+\frac{1}{2}+(-k+2)}, u_{i, j+\frac{1}{2}+(-k+3)}, \dots, u_{i, j+\frac{1}{2}+k})
			\end{align*}
			
			\item Upwind with the sign of $\tilde{v}_{i,j+1}$: 
			\begin{linenomath*}\begin{equation*}
					\hat{u}_{i,j+1} = \begin{cases}
						\hat{u}_{i,j+1}^+ &\mbox{if } \tilde{v}_{i,j+1} < 0 \\
						\hat{u}_{i,j+1}^- &\mbox{if } \tilde{v}_{i,j+1} \ge 0 \\
					\end{cases}
			\end{equation*}\end{linenomath*}
			
			\item Set $ F^{u,y}_{i,j+1} := \tilde{v}_{i,j+1} \hat{u}_{i,j+1} $
		\end{enumerate}
	\end{algorithm}
	
	This completes the description of the scheme \eqref{eq:discrete_simple_mom}. We observe that one can discretize the second equation in \eqref{eq:simple_mom} in an analogous fashion. In particular, the WENO reconstruction of $\partial_y (v^2)$ is
	straightforward and Algorithm \ref{algo:pycles} can be readily modified to yield a WENO reconstruction of $\partial_x (uv)$. 	
	\subsubsection{Order of Accuracy}
	\label{sec:order}
	In this section, we will investigate the (formal) order of accuracy of the WENO scheme described in \cite{Pressel2015}. This order of accuracy is measured in terms of the truncation error of the scheme. We have the following theorem,
	\begin{theorem}
		\label{theorem}
		For $u$, $v$ sufficiently smooth and for any $2 \leq k \in \mathbb{N}$, Algorithm \ref{algo:pycles} produces an at most second order approximation to $\frac{\partial}{\partial y}(uv) (x_i, y_{j+\frac{1}{2}})$ in \eqref{eq:simple_mom}. Hence, the WENO schemes used by \cite{Pressel2015} are at most second order accurate, and equation \eqref{eq:unmetgoal} does not hold for them.
	\end{theorem}
	\begin{proof}
		
		We assume $\Delta x \propto \Delta y$, and so $O(\Delta x) = O(\Delta y)$. Following Algorithm \ref{algo:pycles}, the approximation to $\partial_y (uv) (x_i, y_{j+\frac{1}{2}}) $ is:
		\begin{linenomath*}\begin{equation*}
				D^{u,y}_{i,j+\frac{1}{2}} := \frac{1}{\Delta y} \left(\tilde{v}_{i,j+1} \hat{u}_{i, j+1} - \tilde{v}_{i,j} \hat{u}_{i,j}  \right)
		\end{equation*}\end{linenomath*}
		
		where $\hat{u}$ are the reconstructions of $u$ such that, by \eqref{eq:wenoDiff},    
		\begin{linenomath*}\begin{equation*} 
				\frac{1}{\Delta y} (\hat{u}_{i,j+1} - \hat{u}_{i,j}) = \partial_y u(x_i, y_{j+\frac{1}{2}}) + O(\Delta y^{2k-1}),
		\end{equation*}\end{linenomath*}and $\tilde{v}_{i, m}$ is the horizontally interpolated value for $v$ at $(x_i, y_{m})$, $m \in \{j, j+1\}$, as per step 1 in Algorithm \ref{algo:pycles}:
		\begin{linenomath*}\begin{equation*}
				\tilde{v}_{i,m} =  v(x_i, y_{m}) + O(\Delta x^{2k})
		\end{equation*}\end{linenomath*}

		We will show that $D^{u,y}_{i,j+\frac{1}{2}} = \partial_y (uv) + O(\Delta y^2)$. 
		
		Denote by $v_{i,j+\frac{1}{2}} := v(x_i, y_{j+\frac{1}{2}})$ the exact point value of $v$. Adding and subtracting $v_{i,j+\frac{1}{2}} (\hat{u}_{i,j+1} + \hat{u}_{i,j}) / \Delta y$ and rearranging:
		\begin{linenomath*}\begin{equation} \label{eq:Dlong}
				D^{u,y}_{i,j+\frac{1}{2}} = \frac{\tilde{v}_{i,j+1} - v_{i,j+\frac{1}{2}}}{\Delta y} \hat{u}_{i, j+1}  + \frac{v_{i,j+\frac{1}{2}} - \tilde{v}_{i,j}}{\Delta y} \hat{u}_{i, j} + v_{i,j+\frac{1}{2}} \frac{\hat{u}_{i, j+1} - \hat{u}_{i, j}}{\Delta y}
		\end{equation}\end{linenomath*}
		
		We observe that, by \eqref{eq:wenoDiff}, the last term in the right hand side of \eqref{eq:Dlong} is a high-order $O(\Delta x^{2k-1})$ approximation to $\left(v\,\partial_y u\right) (x_i, y_{j+\frac{1}{2}})$. 
		
		Next, we will show that the other terms in \eqref{eq:Dlong}, which approximate $u\,\partial_y v$, produce a \emph{lower-order} approximation. Using Taylor expansions and \eqref{eq:wenoDiff},
		\begin{align*}
			\frac{\tilde{v}_{i,j+1} - v_{i,j+\frac{1}{2}}}{\Delta y} & \hat{u}_{i, j+1} + \frac{v_{i,j+\frac{1}{2}} - \tilde{v}_{i,j}}{\Delta y} \hat{u}_{i, j} = \\
			& = \frac{v_{i,j+1} - v_{i,j+\frac{1}{2}}}{\Delta y} \hat{u}_{i, j+1} + \frac{v_{i,j+\frac{1}{2}} - v_{i,j}}{\Delta y} \hat{u}_{i, j} + O(\Delta y^{2k-1})\ \\
			&= \left( \frac{\partial_y v(x_i, y_{j+\frac{1}{2}})}{2} + \partial_{yy} v(x_i, y_{j+\frac{1}{2}}) \frac{\Delta y}{8} + O(\Delta y^2)\right) \hat{u}_{i, j+1}\ \\
			&\qquad + \left( \frac{\partial_y v(x_i, y_{j+\frac{1}{2}})}{2} - \partial_{yy} v(x_i, y_{j+\frac{1}{2}}) \frac{\Delta y}{8} + O(\Delta y^2) \right) \hat{u}_{i, j} + O(\Delta y^{2k-1})\ \\
			&= \partial_y v(x_i, y_{j+\frac{1}{2}}) \frac{\hat{u}_{i, j+1} + \hat{u}_{i, j}}{2} + \frac{\partial_{yy} v(x_i, y_{j+\frac{1}{2}})}{8} \frac{\hat{u}_{i,j+1} - \hat{u}_{i,j}}{\Delta y}\Delta y^2 + O(\Delta y^2)\\
			&= \partial_y v(x_i, y_{j+\frac{1}{2}}) \frac{\hat{u}_{i, j+1} + \hat{u}_{i, j}}{2} + \frac{ \left(\partial_{yy} v\ \partial_y u\right) (x_i, y_{j+\frac{1}{2}})}{8} \Delta y^2 + O(\Delta y^2) \\
			&= \partial_y v(x_i, y_{j+\frac{1}{2}}) \frac{\hat{u}_{i, j+1} + \hat{u}_{i, j}}{2} + O(\Delta y^2)
		\end{align*}
		
		For the first term, considering the 1D function $f_i(\cdot) = u(x_i, \cdot)$, and naming $h_{i}(\cdot)$ its 1D WENO reconstruction function as in Proposition \ref{remark:wenoDeriv}, we have that:
		\begin{align} \frac{\hat{u}_{i, j+1} + \hat{u}_{i, j}}{2} &= \frac{1}{\Delta y} \Delta y \frac{h_{i}(y_{j+1}) + h_{i}(y_j)}{2} + O(\Delta y^{2k-1}) \\
			& = \frac{1}{\Delta y} \left( \int_{y_j}^{y_{j+1}} h_{i}(t)~dt + O(\Delta y^3) \right) + O(\Delta y^{2k-1}) \\
			& = u(x_i, y_{j+\frac{1}{2}}) + O(\Delta y^2)
		\end{align}
		Where in the second step we have used the accuracy of the trapezoidal quadrature rule. Finally,
		
		\begin{linenomath*}\begin{equation*}  \frac{\hat{u}_{i, j+1} + \hat{u}_{i, j}}{2} \partial_y v(x_i, y_{j+\frac{1}{2}})= (u\,\partial_y v)(x_i, y_{j+\frac{1}{2}}) + O(\Delta y^2)
		\end{equation*}\end{linenomath*}
		
		Summarizing,
		\begin{linenomath*}	\begin{align*}
				\frac{1}{\Delta y} \left(\tilde{v}_{i,j+1} \hat{u}_{i, j+1} - \tilde{v}_{i,j} \hat{u}_{i,j}  \right) & = \left( (v\,\partial_y u)(x_i, y_{j+\frac{1}{2}}) + O(\Delta y^{2k-1}) \right) + \left( (u\,\partial_y v)(x_i, y_{j+\frac{1}{2}}) + O(\Delta y^2) \right) \\ 
				& = \partial_y (uv) + O(\Delta y^2) \label{eq:qvs2}
		\end{align*}\end{linenomath*}
		
	\end{proof}
	An analogous treatment of the WENO discretization of the $\partial_x (uv)$ term in the second equation of \eqref{eq:simple_mom} also yields an at most second order accuracy of this term. 
	\begin{remark}
		The surprising aspect of Theorem \ref{theorem} is the fact that the formal order of accuracy of the WENO scheme is at most $2$, even if the underlying interpolating polynomials are of arbitrarily high order. As the above proof shows, this can be attributed to the errors that accumulate while interpolating values on the staggered grid. A similar and unexpected loss of convergence is also found for some popular central schemes for anelastic flows, for instance the Wicker-Skamarock scheme presented in \cite{Skamarock2008WRF}), see Appendix \ref{appendixcentral} for details.
	\end{remark}
	\section{An arbitrarily high-order WENO scheme on staggered grids}
	\label{sec:3}
	Given the fact that existing WENO schemes for discretizing the anelastic equations \eqref{eq:momentum} are at most second order accurate, we will present an arbitrarily high-order WENO scheme on such grids. Again for the simplicity of notation and exposition, we will focus on the two-dimensional problem while remarking that the extension to three space dimensions is straightforward.

	\subsection{Interpolation}
	\label{enointerp}
	
	As identified in the previous section, the reason for the loss of accuracy of existing WENO schemes is the error accumulated while interpolating across the staggered grid. In particular, one-dimensional (central) interpolations cause a loss of order of accuracy. A genuinely multi-dimensional interpolation can potentially restore design order of accuracy. Similar ideas were explored in Ghosh and Baeder \cite{Ghosh2012} and Zhang and Jackson \cite{Zhang2009}. In both these papers, the authors consider a non-conservative formulation of the momentum equations that necessitates a genuinely multi-dimensional (central) interpolation. Here, we will retain the conservative version of the momentum equations and design a two-dimensional ENO interpolation, in order to be consistent with the non-oscillatory character of the overall scheme. This increases the computational cost of the scheme, albeit the overall cost is still in the same order as the WENO reconstructions. 
	
	We illustrate the algorithm by describing the discretization of the \emph{problematic} $\partial_y (uv)$ term in the first equation of \eqref{eq:simple_mom}. Recalling the issue from section \ref{sec:2} and as depicted in Figure \ref{fig:2dinterp}, we have to obtain values of $uv$ at a set of points in order to provide a WENO reconstruction of the derivative $\partial_y (uv)$. In particular, we need to find approximations to the values of $v$ at the points $(x_i,y_{j+\frac{1}{2}})$ where values of $u$ are defined. There are two possible choices, as depicted in Figure \ref{fig:2dinterp}. Either we perform a one-dimensional ENO interpolation of $v$ in the $y$-direction to define the point values $v_{i+\frac{1}{2},j+\frac{1}{2}}$ (at the cell centers) and then perform a one-dimensional ENO interpolation of $v_{i+\frac{1}{2},j+\frac{1}{2}}$ along the $x$-direction to obtain the values of $v$ at the points $(x_i,y_{j+\frac{1}{2}})$ (this procedure is shown by solid curves in Figure \ref{fig:2dinterp}) or we perform a one-dimensional ENO interpolation of $v$ in the $x$-direction to define the point values $v_{i,j}$ and then perform a one-dimensional ENO interpolation of $v_{i,j}$ along the $y$-direction to obtain the values of $v$ at the points $(x_i,y_{j+\frac{1}{2}})$ (this procedure is shown by dashed curves in Figure \ref{fig:2dinterp}). In the sequel, we will require the values of the velocity field at cell centers. Therefore, we do the former and summarize the interpolation algorithm below,
	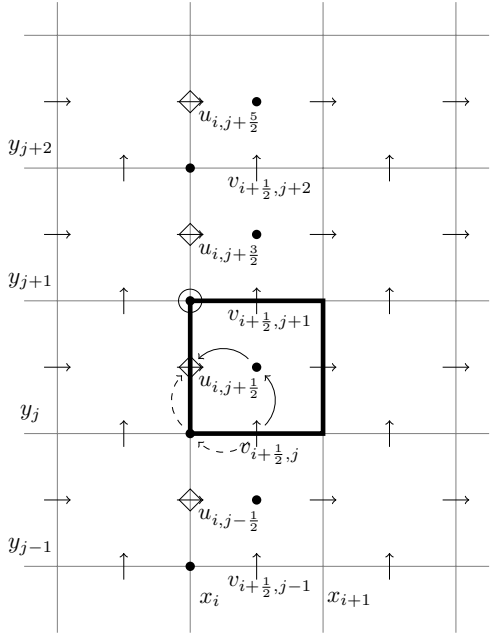
\begin{figure}
		\begin{center}
			\resizebox{0.4\textwidth}{!}{
				\begin{tikzpicture}
				
				\draw[step=2cm,gray,very thin] (-2.5, -1) grid (4.5,8.5);
				\draw[thick, line width=0.7mm] (0,2) rectangle (2,4);
				\draw[->] (-1,-0.2) -- (-1, 0.2);
				\draw[->] (-1,1.8) -- (-1, 2.2);
				\draw[->] (-1,3.8) -- (-1,4.2);
				\draw[->] (-1,5.8) -- (-1,6.2);
				\draw[->] (1,-0.2) -- (1, 0.2);
				\draw[->] (1,1.8) -- (1, 2.2);
				\draw[->] (1,3.8) -- (1,4.2);
				\draw[->] (1,5.8) -- (1,6.2);
				\draw[->] (3,-0.2) -- (3, 0.2);
				\draw[->] (3,1.8) -- (3, 2.2);
				\draw[->] (3,3.8) -- (3,4.2);
				\draw[->] (3,5.8) -- (3,6.2);
				
				\draw[->] (-2.2, 1) -- (-1.8,1);
				\draw[->] (-2.2, 3) -- (-1.8, 3);
				\draw[->] (-2.2,5) -- (-1.8,5);
				\draw[->] (-2.2, 7) -- (-1.8, 7);
				\draw[->] (-0.2, 1) -- (0.2,1);
				\draw[->] (-0.2, 3) -- (0.2, 3);
				\draw[->] (-0.2,5) -- (0.2,5);
				\draw[->] (-0.2, 7) -- (0.2, 7);
				\draw[->] (1.8, 1) -- (2.2,1);
				\draw[->] (1.8, 3) -- (2.2, 3);
				\draw[->] (1.8,5) -- (2.2,5);
				\draw[->] (1.8, 7) -- (2.2, 7);
				\draw[->] (3.8, 1) -- (4.2,1);
				\draw[->] (3.8, 3) -- (4.2, 3);
				\draw[->] (3.8,5) -- (4.2,5);
				\draw[->] (3.8, 7) -- (4.2, 7);
				
				\fill (1,1) circle[radius=2pt];
				\fill (1,3) circle[radius=2pt];
				\fill (1,5) circle[radius=2pt];
				\fill (1,7) circle[radius=2pt];
				\fill (0,0) circle[radius=2pt];
				\fill (0,2) circle[radius=2pt];
				\fill (0,4) circle[radius=2pt];
				\fill (0,6) circle[radius=2pt];
				\node at (0.3,-0.5) {$x_{i}$};
				\node at (2.4,-0.5) {$x_{i+1}$};
				\node at (-2.4, 0.3 ) {$y_{j-1}$};
				\node at (-2.4, 2.3) {$y_{j}$};
				\node at (-2.4,4.3) {$y_{j+1}$};
				\node at (-2.4,6.3) {$y_{j+2}$};
				
				\node at (1.2, -0.3) {$v_{i+\frac{1}{2}, j-1}$};
				\node at (1.2, 1.7) {$v_{i+\frac{1}{2}, j}$};
				\node at (1.2, 3.7) {$v_{i+\frac{1}{2}, j+1}$};
				\node at (1.2, 5.7) {$v_{i+\frac{1}{2}, j+2}$};
				
				\node at (0.6, 0.7) {$u_{i,j-\frac{1}{2}}$};
				\node at (0.6, 2.7) {$u_{i,j+\frac{1}{2}}$};
				\node at (0.6, 4.7) {$u_{i, j+\frac{3}{2}}$};
				\node at (0.6, 6.7) {$u_{i,j+\frac{5}{2}}$};
				
				\node[rectangle,draw,xscale=1, yscale=1,rotate=45] at (0,1) {};
				\node[rectangle,draw,xscale=1, yscale=1,rotate=45] at (0,3) {};
				\node[rectangle,draw,xscale=1, yscale=1,rotate=45] at (0,5) {};
				\node[rectangle,draw,xscale=1, yscale=1,rotate=45] at (0,7) {};
				
				\draw (0, 4) circle[radius=5pt];
				
				\node at (1,2) (1) {};
				\node at (0,2) (2) {};
				\node at (0,3) (4) {};
				\node at (1,3) (3) {};
				
				\draw (1) edge[bend left=45,dashed,->] (2);
				\draw (2) edge[bend left=45,dashed,->] (4);
				\draw (1) edge[bend right=45,->] (3);
				\draw (3) edge[bend right=45,->] (4);
				
				\end{tikzpicture}}
		\end{center}
		\caption{2D interpolation for $\frac{\partial}{\partial y}(uv)$ on a staggered mesh. The two candidate 2-step 1D interpolations are depicted with solid and dashed curved arrows.}
		\label{fig:2dinterp}
	\end{figure}

	\begin{algorithm}
		\label{algo:2dinterp}
		~\textbf{Goal:} Given point values of $v_{i+\frac{1}{2},j} \approx v(x_{i+\frac{1}{2}}, y_j)$, $\forall i,j$, find non-oscillatory, order $2k-1$ interpolations of $v$, at the locations $(x_i, y_{j+\frac{1}{2}})$.
		\begin{enumerate}
			\item \label{step:intermediaterec}For each cell, use 1D vertical ENO interpolation to find an approximation to $v$ at the cell center, $\tilde{v}_{i+\frac{1}{2}, j+\frac{1}{2}}$, based on the stencil of maximal smoothness of $2k-1$ cells contained in
			$\{v_{i+\frac{1}{2}, j+r} \}_{r = -(2k-3)}^{2k-2}$
			\[  \tilde{v}_{i+\frac{1}{2}, j+\frac{1}{2}} = v(x_{i+\frac{1}{2}}, y_{j+\frac{1}{2}}) + O(\Delta y^{2k-1}) \]
			
			\item For each cell, use 1D horizontal ENO interpolation to find an approximation to $v$ at the location of $u$, denote it $\tilde{v}_{i,j+\frac{1}{2}}$, based on the stencil of maximal smoothness of $2k-1$ cells contained in
			$\{ \tilde{v}_{i+\frac{1}{2}+r, j+\frac{1}{2}} \}_{r = -(2k-2)}^{2k-3}$
			\[ \tilde{v}_{i,j+\frac{1}{2}} = v(x_{i}, y_{j+\frac{1}{2}}) + O(\Delta x^{2k-1} + \Delta y^{2k-1})\]
		\end{enumerate}
	\end{algorithm}
	
	\begin{remark}
		The algorithm above can be readily extended to discretize the $\partial_x (uv)$ term in the second equation of \eqref{eq:simple_mom}. This procedure applies almost verbatim to the 3D scenario, the only difference being that a higher number of interpolations, along each of three directions, are required. 
	\end{remark}
	
	\begin{remark} 
		In Algorithm \ref{algo:2dinterp}, we choose to apply 1D ENO interpolation, rather than WENO; this choice is made for stability. \rev{Near sharp gradients, the scheme displays some sensitivity to the reconstruction stencil: at each cell interface, WENO produces two candidate reconstructions, left- and right-biased. We find that a naive upwind choice for this interpolation occasionally produces spurious oscillations. Experimentally, we have found that the best results are produced with ENO, with a stencil that contains at least one node on each side of the point where interpolation is required; this makes the selection unbiased near sharp gradients.} The stencils described above already enforce this requirement. E.g. for step 1, all stencils of $2k-1$ cells contain $v_{i+\frac{1}{2},j}$ and $v_{i+\frac{1}{2},j+1}$.
	\end{remark}
	
	\subsection{Numerical fluxes}
	In this section, we will describe the numerical fluxes for discretizing the advective terms in \eqref{eq:momentum}. As the preceding discussion indicates, we need to differentiate between two sets of terms, namely $\frac{\partial}{\partial x_i} (\rho_0 u_i^2)$, and $\frac{\partial}{\partial x_j} (\rho_0 u_i u_j)$ for $i \neq j$.
	\subsubsection{Terms of the form $\frac{\partial}{\partial x_i}(\rho_0 u_i^2)$}
	\label{1dfluxes}
	
	For illustration, we describe the flux for $\frac{\partial}{\partial x}(\rho_0 u^2)$ in two space dimensions. This configuration is displayed in Figure \ref{fig:uux}, where the diamond marks the position of $u$ in the highlighted cell, and the circles are at the points where fluxes need to be reconstructed.
	
	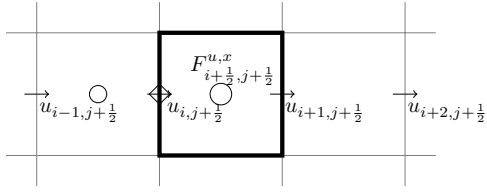
\begin{figure}
		\begin{center}
			\resizebox{0.4\textwidth}{!}{
				\begin{tikzpicture}
				
				\draw[step=2cm,gray,very thin] (-0.5, -0.5) grid (6.5, 2.5);
				\draw[thick, line width=0.7mm] (2,0) rectangle (4,2);
				\draw[->] (-0.2,1) -- (0.2,1);
				\draw[->] (1.8,1) -- (2.2,1);
				\draw[->] (3.8,1) -- (4.2,1);
				\draw[->] (5.8,1) -- (6.2,1);
				\node[rectangle,draw,xscale=1, yscale=1,rotate=45] at (2,1) {};
				\draw (3,1) circle[radius = 5pt];
				\draw (1,1) circle[radius = 4pt];
				\node at (0.7, 0.7) {$u_{i-1, j+\frac{1}{2}}$};
				\node at (2.6, 0.7) {$u_{i, j+\frac{1}{2}}$};
				\node at (4.7, 0.7) {$u_{i+1, j+\frac{1}{2}}$};
				\node at (6.7, 0.7) {$u_{i+2, j+\frac{1}{2}}$};
				\node at (3.2, 1.4) {$F^{u,x}_{i+\frac{1}{2},j+\frac{1}{2}}$};
				\end{tikzpicture}}
		\end{center}
		\caption{Reconstruction of $\frac{\partial}{\partial x} u^2$ on a staggered grid. The diamond marks the location of $u$ in the highlighted cell; numerical fluxes must be calculated at the circles.}
		\label{fig:uux}
	\end{figure}

	\begin{algorithm}
		\label{algo:squaredfluxes}
		~\textbf{Goal:} Find a numerical flux $F^{u,x}$ such that, formally, $\forall i,j$:
		
		\begin{linenomath*}\begin{equation*}
				\frac{F^{u,x}_{i+\frac{1}{2}, j+\frac{1}{2}} - F^{u,x}_{i-\frac{1}{2},j+\frac{1}{2}}}{\Delta x} = \frac{\partial (\rho_0 u^2)}{\partial x} (x_i, y_{j+\frac{1}{2}}) + O(\Delta x^{2k-1} + \Delta y^{2k-1})
		\end{equation*}\end{linenomath*}

		\begin{enumerate}
			\item Compute the two biased WENO reconstructions of the flux function $f(x) = \left(\rho_0 u^2\right) (x,y_{j+\frac{1}{2}})$ at $(x_{i+\frac{1}{2}}, y_{j+\frac{1}{2}})$. Denoting $f_l :=  \rho_0(x_{l}, y_{j+\frac{1}{2}})~ u_{l, j+\frac{1}{2}}^2$, this is: 
			\begin{align*} \hat{f}_{i+\frac{1}{2}, j+\frac{1}{2}}^- & = W_{+\frac{1}{2}}( f_{i-(k-1)}, f_{i-(k-2)}, \dots, f_{i+(k-1)} ) \\
				\hat{f}_{i+\frac{1}{2}, j+\frac{1}{2}}^+ & = W_{-\frac{1}{2}}( f_{i-(k-2)}, f_{i-(k-1)}, \dots, f_{i+k} ) 
			\end{align*}
			
			\item Choose a consistent numerical flux function $G(u^-, u^+)$, and set
			\[ F^{u,x}_{i+\frac{1}{2}, j+\frac{1}{2}} := G(\hat{f}_{i+\frac{1}{2}, j+\frac{1}{2}}^-, \hat{f}_{i+\frac{1}{2}, j+\frac{1}{2}}^+) \]
			
		\end{enumerate}
	\end{algorithm}
	
	\begin{remark}
		
		In the numerical experiments in the sequel, $G$ is taken to be the upwind flux:
		\[ G(u_l,u_r) = \begin{cases}
		u_l &\mbox{if } a_{i + \frac{1}{2}, j + \frac{1}{2}} \ge 0 \\
		u_r &\mbox{if } a_{i + \frac{1}{2}, j + \frac{1}{2}} < 0
		\end{cases}; \qquad a_{i + \frac{1}{2}, j + \frac{1}{2}} = \dfrac{u_{i,j+\frac{1}{2}} + u_{i+1,j+\frac{1}{2}}}{2}.
		\]
		Other fluxes, e.g. Lax-Friedrichs, Rusanov, etc. can also be considered. 
	\end{remark}
	
	\subsubsection{Terms of the form $\frac{\partial}{\partial x_j} (\rho_0 u_i u_j)$, with $i \neq j$}
	\label{crossfluxes}
	Again for the purpose of illustration, we focus on the two-dimensional discretization of the term $\frac{\partial}{\partial y} (\rho_0 uv)$. As discussed before, this computation requires interpolations of $v$ at the points where $u$ lies; i.e. this algorithm assumes the availability in each cell of interpolations \[ \tilde{v}_{i, j+\frac{1}{2}+m} := v(x_{i}, y_{j+\frac{1}{2}+m}) + O(\Delta x^{2k-1} + \Delta y^{2k-1}),\quad \text{ for } m \in \{-(k-1), \dots, k\},\] which can be computed with Algorithm \ref{algo:2dinterp}. This is depicted in Figure \ref{fig:uvx}, where the diamond marks the position of $v$ in the highlighted cell; and the circles are at the locations where fluxes need to be computed.

	\begin{figure}
		\begin{center}
			\resizebox{0.18\textwidth}{!}{
				\begin{tikzpicture}
				\draw[step=2cm,gray,very thin] (-0.5, -0.5) grid (2.5, 8.5);
				\draw[thick, line width=0.7mm] (0,2) rectangle (2,4);
				\node[rectangle,draw,xscale=1.5, yscale=1.5,rotate=45] at (0,3) {};
				\draw (0,2) circle[radius = 5pt];
				\draw (0,4) circle[radius = 6pt];
				\draw[->] (1,-0.2) -- (1,0.2);
				\draw[->] (1,1.8) -- (1, 2.2);
				\draw[->] (1, 3.8) -- (1, 4.2);
				\draw[->] (1, 5.8) -- (1, 6.2);
				\draw[->] (-0.2, 1) -- (0.2, 1);
				\draw[->] (-0.2, 3) -- (0.2, 3);
				\draw[->] (-0.2, 5) -- (0.2, 5);
				\draw[->] (-0.2,7) -- (0.2,7);
				\fill (0,1) circle[radius = 3pt];
				\fill (0,3) circle[radius = 3pt];
				\fill (0,5) circle[radius = 3pt];
				\fill (0,7) circle[radius = 3pt];
				\node at (1, 7) {$(u \tilde{v})_{i, j+\frac{5}{2}}$};
				\node at (1, 5) {$(u \tilde{v})_{i, j+\frac{3}{2}}$};
				\node at (1, 3) {$(u \tilde{v})_{i, j+\frac{1}{2}}$};
				\node at (1, 1) {$(u \tilde{v})_{i, j-\frac{1}{2}}$};
				\node at (0.6, 3.6) {$F^{u,y}_{i,j+1}$};
				\end{tikzpicture} }
			
		\end{center}
		\caption{Reconstruction of $\frac{\partial}{\partial y} (uv)$ on a staggered grid. The diamond marks the location of $u$ in the highlighted cell; numerical fluxes must be calculated at the circles. Interpolations of $v$ are required at the points marked with dots.}
		\label{fig:uvx}
	\end{figure}
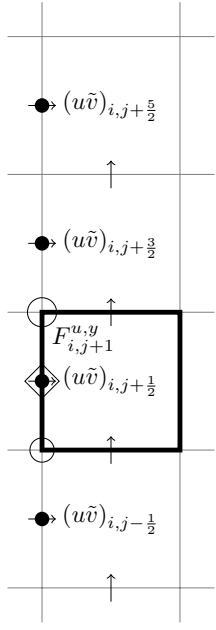
	
	With said interpolations, we now have high-order approximations to the values of $(uv)$ at the desired points (marked with black dots in Fig. \ref{fig:uvx}). Thus, for $F^{u,y}_{i, j+1}$, the procedure would be:
	
	\begin{algorithm}
		\label{algo:crossfluxes}
		~\textbf{Goal:} Find a numerical flux $F^{u,y}_{i, j+1}$ such that, formally, $\forall i,j$:
		
		\begin{linenomath*}\begin{equation*}
				\frac{F^{u,y}_{i, j+1} - F^{u,y}_{i,j}}{\Delta y} = \frac{\partial (\rho_0 uv)}{\partial y} (x_i, y_{j+\frac{1}{2}}) + O(\Delta x^{2k-1} + \Delta y^{2k-1})
		\end{equation*}\end{linenomath*}
		
		\begin{enumerate}
			
			\item Compute the two biased WENO reconstructions of the flux function $g(y) = \left(\rho_0 uv\right) (x_i, y)$ at $(x_{i}, y_{j+1})$. Denoting $g_l :=  \rho_0(x_{i}, y_{l+\frac{1}{2}}) u_{i, l+\frac{1}{2}} \tilde{v}_{i,l+\frac{1}{2}}$, with $\tilde{v}_{i,l+\frac{1}{2}}$ the high-order interpolations of Algorithm \ref{algo:2dinterp}, results in 
			\begin{align*} \hat{g}_{i, j+1}^- & = W_{+\frac{1}{2}}( g_{j-(k-1)}, g_{j-(k-2)}, \dots, g_{j+(k-1)} ) \\
				\hat{g}_{i, j+1}^+ & = W_{-\frac{1}{2}}( g_{j-(k-2)}, g_{j-(k-1)}, \dots, g_{j+k} ) 
			\end{align*}
			
			\item Choose a numerical flux function $G(u^-, u^+)$, and set
			\[ F^{u,y}_{i, j+1} := G(\hat{g}_{i, j+1}^-, \hat{g}_{i, j+1}^+). \]
			Again, one can use the upwind flux function and upwind with respect to $a_{i, j+\frac{1}{2}} := \tilde{v}_{i,j+\frac{1}{2}}$, or use a Lax-Friedrichs or Rusanov flux; the numerical examples shown in the sequel use upwind.
			
		\end{enumerate}
	\end{algorithm}
	\subsection{The advection scheme}
	We collect all the above ingredients and write the advection scheme as:
	
	\begin{algorithm}
		\label{algo:full}
		~\textbf{Goal:} Given $i \in \{1,\dots,d\}$, obtain a high-order, essentially non-oscillatory approximation, on an Arakawa-C grid (Fig. \ref{fig:staggeredgrid}), to the spatial derivatives in:
		\[ \frac{\partial}{\partial t} u_i = -\frac{1}{\rho_0} \sum_{j=1}^d \frac{\partial}{\partial x_j} (\rho_0 u_i u_j) \]
		
		\begin{enumerate}
			\item Using Algorithm \ref{algo:2dinterp}, $\forall j \in \{1,\dots,d\}$, $j \neq i$, obtain high-order approximations to $u_j$ at the location of $u_i$.
			
			\item For $j \in \{1,\dots,d\}$: \begin{enumerate}
				\item If $j = i$, find a numerical flux $F^{u_i, x_i}$ for $\frac{\partial}{\partial x_i} (\rho_0 u_i^2)$ using Algorithm \ref{algo:squaredfluxes}.
				\item Otherwise, find a numerical flux $F^{u_i, x_j}$ for $\frac{\partial}{\partial x_j} (\rho_0 u_i u_j)$ using Algorithm \ref{algo:crossfluxes}.
			\end{enumerate}
			
			\item Find the spatial differences of the computed numerical fluxes $F^{u_i, x_k}$ in the direction of $x_k$, $k \in \{1, \dots, d\}$, and divide by $\rho_0$ to complete the approximation. 
			As an example, the semi-discrete scheme for discretizing \eqref{eq:simple_mom} is 
			\begin{equation}
				\label{eq:2ds}
				\begin{aligned}
					\frac{\partial}{\partial t} u_{i,j+\frac{1}{2}} &= -\frac{1}{\rho_0(x_i, y_{j+\frac{1}{2}})}\left( \frac{F^{u,x}_{i+\frac{1}{2}, j+\frac{1}{2}} - F^{u,x}_{i-\frac{1}{2},j+\frac{1}{2}}}{\Delta x}
					+ \frac{F^{u,y}_{i, j+1} - F^{u,y}_{i,j}}{\Delta y}
					\right) \\
					\frac{\partial}{\partial t} v_{i+\frac{1}{2},j} &= -\frac{1}{\rho_0(x_{i+\frac{1}{2}}, y_j)}\left( \frac{F^{v,y}_{i+1, j} - F^{v,y}_{i,j}}{\Delta x}
					+ \frac{F^{v,y}_{i+\frac{1}{2}, j+\frac{1}{2}} - F^{u,y}_{i+\frac{1}{2},j-\frac{1}{2}}}{\Delta y}
					\right) 
				\end{aligned}
			\end{equation}
		\end{enumerate}
	\end{algorithm}

	\begin{lemma} \label{lemma}
		For sufficiently smooth $U$, $\rho_0$, Algorithm \ref{algo:full} produces arbitrarily high-order accurate approximations to the spatial derivatives of the fluxes in \eqref{eq:momentum}, i.e, for any $k \in {\mathbb N}$, the truncation error in discretizing the advective parts of \eqref{eq:momentum} is $O(\Delta x^{2k-1} + \Delta y^{2k-1})$.
	\end{lemma}
	The proof follows in a straightforward manner from the construction and for the accuracy of the WENO method, see Proposition \ref{remark:wenoDeriv}. 
	
	\subsection{The complete scheme}
	In this section, we will provide an overview of the discretizations of other terms in the anelastic equations.
	
	\subsubsection{Source terms}
	Source terms in the equations (e.g. buoyancy $b$ in the equation of momentum \eqref{eq:momentum}, or any other physical source terms) are easy to deal with in this context. Since buoyancy
	has a closed form in the anelastic formulation (see section \ref{passivescalar}), this reduces to computing the source terms at the appropriate locations and adding them to the tendencies, i.e, discretizations of the advective terms, computed in Algorithm \ref{algo:full}.
	
	\subsubsection{Pressure}
	For the approximation of the pressure gradient in \eqref{eq:momentum}, and the continuity equation \eqref{eq:continuity}, the well-known idea of Leray projection into the space of divergence-free fields can be used. The divergence operator is applied to \eqref{eq:momentum} yielding,
	\[ \sum_{i=1}^d \frac{\partial^2}{\partial t \partial x_i} \rho_0 u_i = \sum_{i=1}^d \sum_{j=1}^d -\frac{\partial^2}{\partial x_i x_j} (\rho_0 u_i u_j) +   \frac{\partial}{\partial x_3}(\rho_0 b \delta_{i,3}) - \sum_{i=1}^d \frac{\partial}{\partial x_i} \rho_0 \frac{\partial}{\partial x_i} \frac{p'}{\rho_0}. \]
	Discretizing the time derivative with forward Euler method with time step $\Delta t$, we can rewrite it as:
	\[ \sum_{i=1}^d \left( \frac{\partial}{\partial x_i} \rho_0 u_i \right)^{n+1} = \sum_{i=1}^d \left( \frac{\partial}{\partial x_i} \rho_0 u_i \right)^{n} + \Delta t \sum_{i=1}^d \left(\sum_{j=1}^d \frac{\partial^2}{\partial x_i x_j} (\rho_0 u_i u_j) + \frac{\partial}{\partial x_3}(\rho_0  b \delta_{i,3}) - \frac{\partial}{\partial x_i} \rho_0 \frac{\partial}{\partial x_i} \frac{p'}{\rho_0} \right) \]
	
	Setting the left hand side to zero, in order to enforce null divergence, and rearranging results in
	\begin{equation} \label{eq:poissonlike} \sum_{i=1}^d \frac{\partial}{\partial x_i} \rho_0 \frac{\partial}{\partial x_i} \frac{p' \Delta t}{\rho_0} = \sum_{i=1}^d \frac{\partial}{\partial x_i} \rho_0 \left(  u_i^n + \frac{\Delta t}{\rho_0} \sum_{j=1}^d\frac{\partial}{\partial x_j}(\rho_0 u_i u_j) + \Delta t b \delta_{i,3} \right) \end{equation}
	
	We observe that eq. \eqref{eq:poissonlike} is a Poisson equation, with variable coefficients, for the function $\frac{p' \Delta t}{\rho_0}$, and the right hand side is the result of evolving $(u_i)^n$ by forward Euler with the spatial tendencies computed with Algorithm \ref{algo:full}.
	\rev{
	The approach that we use for solving this equation is standard in the literature; we include a sketch of the algorithm here for completeness. Let us consider for simplicity problems depending only on the horizontal directions $(x, y)$, in which we assume periodic boundary conditions, and constant $\rho_0$. If we denote by $\Delta$ the spatial Laplace operator, and $\tilde{U}$ the vector field of velocities computed by Algorithm \ref{algo:full}, 
	\[ \tilde{U}_i := u_i^n + \frac{\Delta t}{\rho_0} \sum_{j=1}^d\frac{\partial}{\partial x_j}(\rho_0 u_i^n u_j^n) + \Delta t b^n \delta_{i,3} , \qquad \text{ for } i=\{1, \dots, d\}, \]
	the scheme reduces, up to constant factors, to
	\begin{equation} \Delta p' = \nabla \cdot \tilde{U}. \label{eq:simplelapl} \end{equation}
	}
	\begin{algorithm}
	\label{algo:pressure}
	\rev{
	\textbf{Goal:} Given a field of intermediate values $\tilde{U}_{i,j}$, and a one-dimensional discrete approximation to the derivative, operator $D$, compute the projection $U_{i,j}$ of $\tilde{U}$ into the space of discretely divergence free functions with respect to $D$. That is, for all $i,j$,
	\[ D_1 (U_1)_{i,j} + D_2 (U_2)_{i,j} = 0 \]
	}
	
	\begin{enumerate}
	    \item \label{step:div} For each cell $C_{i,j}$, compute an approximation to the divergence $\eta_{i,j} = D_1 (\tilde{U}_1)_{i,j} + D_2 (\tilde{U}_2)_{i,j} $.
	    \item Apply a discrete Fourier transform to $\{\eta\}_{i,j}$ to obtain its coefficients $\hat{\eta}_{k,l}$.
	    \item Some straightforward computations on eq. \eqref{eq:simplelapl} show that the Fourier coefficients of $p$ are given by 
	    \[ \hat{p}_{k,l} := -\frac{1}{4 \pi^2} \frac{\hat{\eta}_{k,l}}{\frac{k^2}{(x_{max}-x_{min})^2} + \frac{l^2}{(y_{max}-y_{min})^2} } \]
	    \item \label{step:grad} Apply the inverse discrete Fourier transform to compute the values of $p_{i,j}$ at the grid, and its gradient, $G_{i,j} = (D_1 p_{i,j}, D_2 p_{i,j})$. 
	    \item Set $U_{i,j} := \tilde{U}_{i,j} - G_{i,j}$.
	\end{enumerate}
	
	\end{algorithm}
	
	\begin{remark}
	\rev{
	\label{rem:discrdiv}
	Choosing the discrete derivative operator $D$ presents some difficulties. We wish our scheme to be discretely divergence-free with respect to the WENO derivative operator used in Algorithm \ref{algo:full}. However, its exact form depends on the data, through the stencil selection process. If we choose $D$ to be the WENO derivative, due to different choices of stencils in steps \ref{step:div} and \ref{step:grad}, it cannot be guaranteed that $D_1 (U_1)_{i,j} + D_2 (U_2)_{i,j} = 0$. Therefore, we opt for a high-order central difference operator in the divergence-projection step. Although the divergence is only approximated up to truncation errors, we empirically obtain good results from this choice, even in the presence of sharp gradients. Some examples can be found in section \ref{sec:4}.
	}
	\end{remark}
	
	\begin{remark}
	\rev{
	For other, non-periodic boundary conditions, the treatment is more involved. For numerical example \ref{bubble}, we impose no-flow boundary conditions across the vertical boundaries. The lack of periodicity means that the efficient, pure Fourier algorithm above is no longer applicable. Instead, we combine one-dimensional Fourier transforms for the 1D horizontal cuts, and finite differences with symmetric boundary conditions for derivatives in the vertical direction. For a low-order approximation to the second derivative, this can be formulated as a semi-discretization ($x$-continuous, $y$-discrete), such as:
	\begin{align} \begin{split} \label{eq:nonperiodicpres}  
	\rho_0(y_{j+\frac{1}{2}}) \partial_{xx} p_{j+\frac{1}{2}}(x_{i+\frac{1}{2}}) + \frac{\rho_0(y_{j+1})[p_{j+\frac{3}{2}}(x_{i+\frac{1}{2}}) - p_{j+\frac{1}{2}}(x_{i+\frac{1}{2}})] - \rho_0(y_j)[p_{j+\frac{1}{2}}(x_{i+\frac{1}{2}}) - p_{j-\frac{1}{2}}(x_{i+\frac{1}{2}})]}{\Delta y^2}\\ \qquad = \eta_{j+\frac{1}{2}}(x_{i+\frac{1}{2}}), \end{split} \end{align}
	where in the case of constant $\rho_0(y)$, the difference above reduces to a standard three-point approximation of the second derivative. Applying a Fourier transform (in the horizontal dimension) to equation \eqref{eq:nonperiodicpres}, and exploiting linearity, will ultimately produce the Fourier coefficients of the 1D cuts as a solution of a (tridiagonal, in this case) linear system of equations. Further details can be found e.g. in \cite{Pressel2015} and references therein.
	}
	\end{remark}

	\subsubsection{Time evolution}
	So far, we have a semi-discrete scheme of the form:
	\[ \frac{d}{dt} U_{i,j}(t) = L(t, U(t)) \]
	where $L$ is the operator that collects all spatial discretizations. A time marching algorithm is necessary to complete our discretization. Given our focus on obtaining a non-oscillatory discretization, we choose the popular strong stability preserving, third order Runge-Kutta method (SSP-RK3), described e.g. by Shu in \cite{Shu1989} for discretizing ODE $\varphi^{\prime}(t) = L(t,\varphi(t))$ with 
	
	\begin{align*}
		\varphi^{(1)} & = \varphi_n + \Delta t (L(t_n, \varphi_n)) \\
		\varphi^{(2)} & = \frac{3}{4} \varphi_n + \frac{1}{4}\left( \varphi^{(1)} + \Delta t \left(L(t_n + \Delta t, \varphi^{(1)}) \right) \right) \\
		\varphi_{n+1} & = \frac{1}{3} \varphi_n + \frac{2}{3}\left( \varphi^{(2)} + \Delta t \left(L(t_n + \frac{\Delta t}{2}, \varphi^{(2)}) \right) \right)
	\end{align*}	
	\subsubsection{Layout of the full scheme}
	The complete scheme is realized in the following algorithmic form, 
	
	\begin{algorithm}
		\label{algo:everything}
		~\textbf{Goal:} Find numerical approximations to the anelastic Euler equations \eqref{eq:momentum}-\eqref{eq:continuity} on an Arakawa-C grid.
		
		For each time step, $t = t^n$: \begin{enumerate}
			\item Compute $\Delta t^n$ such that a CFL condition is verified
			\item For each Runge-Kutta substep: \begin{enumerate}
				\item Compute the advective tendencies for momentum (Algorithm \ref{algo:full})
				\item Add source terms
				\item Compute dynamic pressure and apply correction to computed advective tendencies (Algorithm \ref{algo:pressure})
			\end{enumerate}
			\item Compute total SSP-RK3 tendencies and update velocities
			\item $t^{n+1} = t^n + \Delta t^n$
		\end{enumerate}
	\end{algorithm}

	\section{Numerical experiments}
	\label{sec:4}
	In this section, we will present a suite of numerical experiments in order to compare WENO schemes, such as the one described in \cite{Pressel2015} and the one proposed here, to each other and to well-known central schemes such as the	Wicker-Skamarock scheme \cite{Skamarock2008WRF} and the Morinishi scheme \cite{Morinishi1998}. See Appendix \ref{appendixcentral} for a detailed description of these schemes.
	\subsection{Discontinuous vortex patch}
	\label{sec:vortexpatch}
	The first numerical experiment is designed to test the ability of a finite difference scheme to approximate sharp gradients. We consider only the advective part of the momentum equations \eqref{eq:simple_mom} on
	the two-dimensional domain $D= [0,2\pi]^2$ with periodic boundary conditions. The initial datum is
	\begin{linenomath*}\begin{equation*}
			u_0(x,y) = \begin{cases}
				-\frac{1}{2}(y-\pi)  &\mbox{if } (x,y)\in\Delta \\
				0  &\mbox{otherwise}
			\end{cases},\hspace{1cm}
			v_0(x,y) = \begin{cases}
				\frac{1}{2}(x-\pi)  &\mbox{if } (x,y)\in\Delta \\
				0  &\mbox{otherwise} \\
			\end{cases}
	\end{equation*}\end{linenomath*}
	
	Here, $\Delta = \{(x,y)\in\mathbb{R}^2, (x-\pi)^2 + (y-\pi)^2 < \frac{\pi}{2}\}$. Thus, initial velocity is discontinuous and the initial vorticity is confined to the region $\Delta$. All schemes are tested on a uniform Cartesian
	grid of $128^2$ points. The time step is chosen to be consistent with a CFL number of $0.1$. 
	
	We approximate \eqref{eq:simple_mom} with the sixth-order scheme of Morinishi et al \cite{Morinishi1998}, see also Appendix A, on this grid and display the results at time $T=1$ in Figure \ref{fig:vortexpatchm6}. We observe from Figure \ref{fig:vortexpatchm6} that the 
	computed velocity with this high-order central scheme is already very oscillatory around the discontinuity at this time. These spurious oscillations build up in time and the scheme blows up by time $T=2$. Similar behavior was also seen in the case of the Wicker-Skamarock scheme, \cite{Skamarock2008WRF} and Appendix \ref{appendixcentral}, and for different orders of the central schemes. The numerical results are expected as the central scheme will become unstable around discontinuities. 
	
	We compute solutions for this test case with the fifth order versions of the WENO scheme of \cite{Pressel2015} and the WENO scheme proposed in this paper, i.e. Algorithm \ref{algo:everything}; and display the results in Figures \ref{fig:vortexpatchorig} and \ref{fig:vortexpatchnew}. The results show that both schemes produce very stable approximations, even at the time $T=5$. In particular, the vortex patch is very well confined to the disk. There are minor differences between the WENO scheme of \cite{Pressel2015} that we denote by Algorithm \ref{algo:pycles} and the scheme proposed here (Algorithm \ref{algo:everything}). In particular, although both schemes produce small amplitude oscillations at the discontinuity, the WENO scheme proposed here resolves the discontinuity with oscillations of significantly smaller amplitude. 
	
	Summarizing, both WENO schemes are quite good at resolving discontinuities or sharp gradients. This is in sharp contrast to central finite difference scheme which can blow up at discontinuities due to spurious oscillations.
	
	\begin{figure}[htp]
		\centering
		\begin{subfigure}{.33\textwidth}
			\centering
			\includegraphics[width=\linewidth,clip=true]{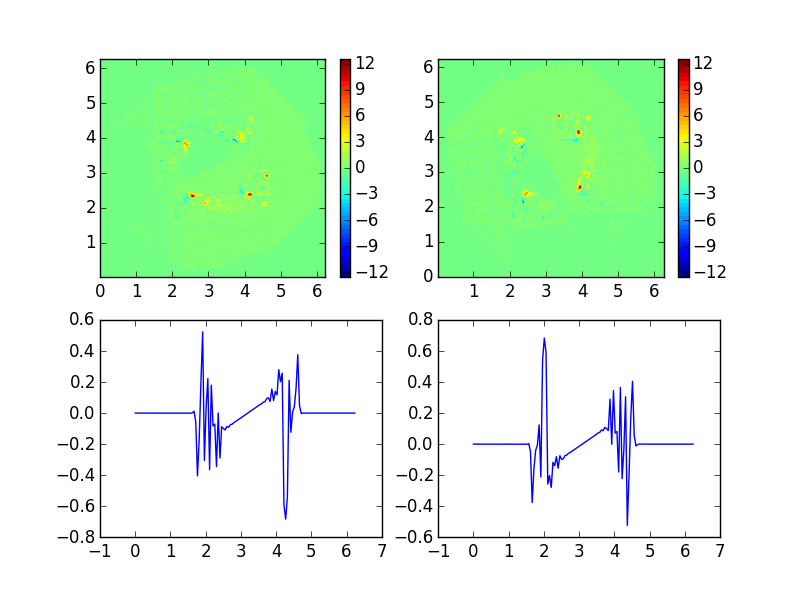}
			\caption{6th order Morinishi, $T = 1$}
			\label{fig:vortexpatchm6}
		\end{subfigure}%
		\begin{subfigure}{.33\textwidth}
			\centering
			\includegraphics[width=\linewidth,clip=true]{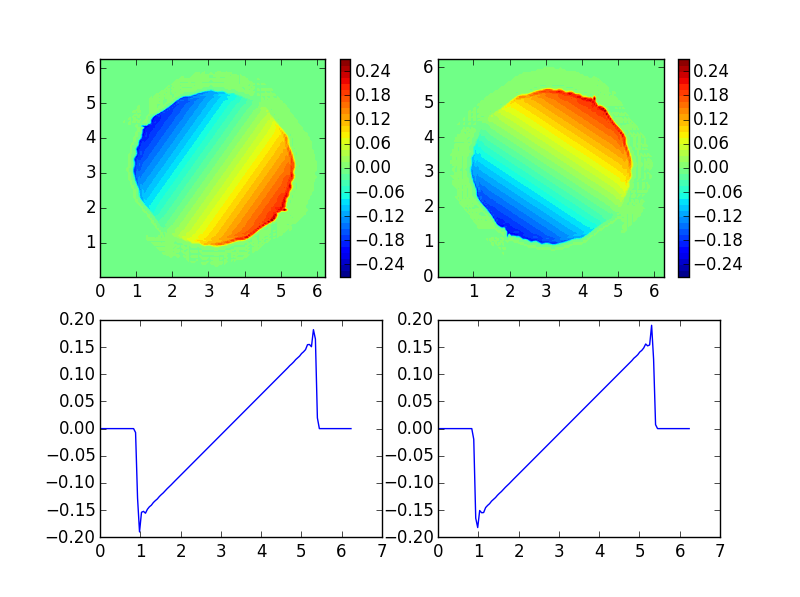}
			\caption{Algorithm \ref{algo:pycles}), $T=5$}
			\label{fig:vortexpatchorig}
		\end{subfigure}%
		\begin{subfigure}{.33\textwidth}
			\centering
			\includegraphics[width=\linewidth,clip=true]{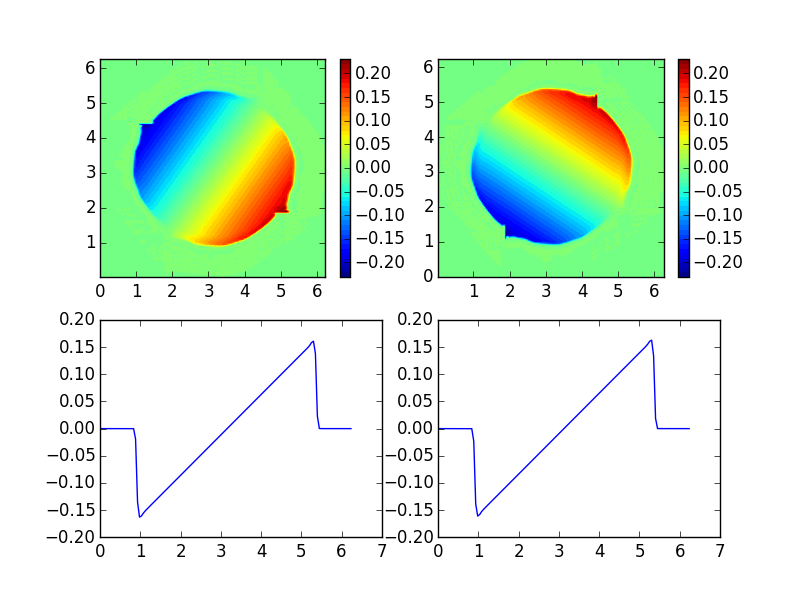}
			\caption{Algorithm \ref{algo:full}), $T=5$}
			\label{fig:vortexpatchnew}
		\end{subfigure}
		\caption{Simulations of discontinuous vortex patch problem. In each figure, left plots represent $u$ (2D, and 1D cut along $y=\pi$), right ones $v$ (2D, and 1D cut along $x = \pi$).}
	\end{figure}
	\subsection{Taylor vortex}
	\label{taylorvortex}
	In this experiment, we fix the domain $[-8,8]^2$ with periodic boundary conditions and solve the anelastic equations \eqref{eq:momentum}, \eqref{eq:continuity} with buoyancy $b=0$ and $\rho_0 \equiv 1$, i.e. incompressible Euler, and the following initial condition,
	\begin{linenomath*}\begin{align}
			u_0(x,y) & = -y e^{\frac{1}{2}(1 - x^2 - y^2)}  + 8 \\
			v_0(x,y) & = \ \ x e^{\frac{1}{2}(1 - x^2 - y^2)}
	\end{align}\end{linenomath*}
	As no exact solution is available, we compute a reference solution on a very fine grid of $8748 \times 8748$ points, and an appropriately small time step, using a sixth-order Morinishi scheme of \cite{Morinishi1998}. We will compute approximations with the WENO scheme of \cite{Pressel2015} (Algorithm \ref{algo:pycles}) and the WENO scheme proposed here (Algorithm \ref{algo:everything}) and compute corresponding errors with respect to the reference solution. 
	
	For all solvers discussed in this test, we use a pressure solver based on high-order, centered finite difference approximations to the divergence, as stated in Remark \ref{rem:discrdiv}
	
	To guarantee that the leading error term in the advection scheme stems from the spatial discretization, we require $\Delta x^q > \Delta t^3$ (as we use SSP-RK3). For this reason, we limit these tests to $q \in \{ 3,5 \}$ (i.e. $k=\{2,3\}$ in Algorithms \ref{algo:pycles} and \ref{algo:everything}), and we take a fixed $\Delta t$ with $\Delta t^3 < \Delta x^5$. For the resolutions considered, $\Delta t = 0.0001$ suffices. We run our tests up to $T = 0.01$ (i.e. 100 time steps).	
	For both the original and new schemes, we have compared third and fifth order approximations, on grids where $n_x = n_y$. We show here the discrete $L^1$ error for $u$ with respect to the reference solution; \rev{across all numerical examples, we find that the behaviour of the error is completely analogous in all discrete $L^p$-norms, for $p \in [1, \infty)$.}
	
	\begin{table}
		\begin{center}
			\begin{tabular}{c || c c | c c || c c | c c }
				& \multicolumn{4}{c}{Advection scheme (Alg. \ref{algo:pycles})} & \multicolumn{4}{c}{High-order formulation (Alg. \ref{algo:full})}\\ \hline
				$n_x$ & u (3rd) & EOC & u (5th) & EOC  & u (3rd) & EOC & u (5th) & EOC \\ \hline
				12 & 0.0469 & -- &  0.0424 & -- & 0.0529 & -- &  0.0465 & --   \\
				36 & 0.0267 & 0.51 & 0.00132 & 3.15 & 0.0069 & 1.66  & 0.00167 & 3.03  \\
				108 & 0.00802 & 1.09 & 8.41e-5 & 2.51  & 0.000419 & 2.55 & 7.45e-6 & 4.93 \\
				324 & 0.00284 & 0.94 & 9.89e-6 & 1.94 & 2.67e-5 & 2.51 & 3.89e-8 & 4.78 \\
				972 & 0.000939 & 1.01 & 1.06e-6 & 2.03  & 1.39e-5 & 2.69 & 1.88e-10 & 4.85 \\
			\end{tabular}
		\end{center}
		\caption{Discrete $L^1$ numerical error for experiment \ref{taylorvortex} (Taylor vortex)}
		\label{tab:taylorvtx}
	\end{table}
	
	The numerical error results in Table \ref{tab:taylorvtx} are consistent with our theoretical findings. The new scheme converges at a rate that is dictated by the order of the underlying approximations whereas the WENO scheme used by \cite{Pressel2015} is restricted to at most second order accuracy. Interestingly, the version of Algorithm \ref{algo:pycles} using third order interpolations shows only first order convergence. This is on account of the choice of the stencil in step 1 of Algorithm \ref{algo:pycles}: instead of taking $2k$ points, in the description of the scheme in \cite{Pressel2015}, the $(2k-2)$-point symmetric stencil was chosen. Revisiting the proof of Theorem \ref{theorem}, it is easy to see that taking a $(2k-2)$-nd order accurate approximation, the order of convergence is $\min\{2, 2k-3\}$; with $k=2$ (i.e. WENO3) this only produces first-order accuracy. 		
	
	The computational cost, for different advection schemes on this test problem, is shown in Figure \ref{fig:compcost}. Fig. \ref{fig:runtime} shows that Algorithm \ref{algo:everything} is the most expensive among all the schemes, on account of the many high-order interpolations that are required herein. However, this cost pays off when we consider the overall computational complexity of the competing schemes i.e, the error with respect to the computational cost. This complexity is plotted in figure \ref{fig:error}. From this figure, we observe that clearly the arbitrarily high-order WENO schemes proposed here significantly outperform the WENO schemes of \cite{Pressel2015}. For the same cost, the error with the genuinely high-order schemes is indeed significantly less than the WENO schemes of \cite{Pressel2015}. Moreover, the proposed WENO schemes are comparable in the error vs. cost with corresponding high-order central schemes of \cite{Morinishi1998}. This is not surprising as these central finite difference schemes are also high-order accurate. On the other hand, given the fact that the central schemes will oscillate at sharp gradients, the high-order WENO schemes provide an added advantage over them. Figure \ref{fig:error} also nicely illustrates the advantage of using a high-order scheme: one obtains a much smaller error at comparable cost to a low-order scheme, even if a low-order scheme has low cost on the same resolution. 
	
	\begin{figure}[htp]
		\centering
		\begin{subfigure}{.5\textwidth}
			\centering
			\includegraphics[width=\linewidth,clip=true]{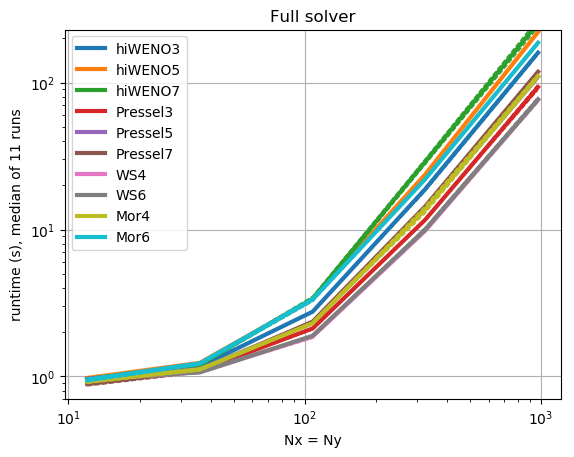}
			\caption{Total runtime of simulation}
			\label{fig:runtime}
		\end{subfigure}%
		\begin{subfigure}{.5\textwidth}
			\centering
			\includegraphics[width=\linewidth,clip=true]{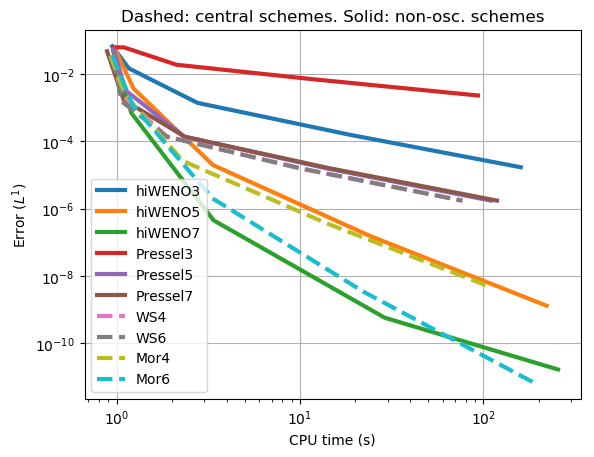}
			\caption{Numerical error for a given runtime}
			\label{fig:error}
		\end{subfigure}%
		\caption{Study of the computational cost of the full scheme, when using different algorithms for the advection: Alg. \ref{algo:full} (labeled hiWENO), Alg. \ref{algo:pycles} (labelled Pressel), as well as Wicker-Skamarock and Morinishi (see Appendix \ref{appendixcentral}).}
		\label{fig:compcost}
	\end{figure}

	\subsection{Double shear layer}
	\label{shear}
	In this numerical example, we present the well-known test of the (double) shear layer, e.g. \cite{Zhang2009}, \cite{Ghosh2012}.
	
	For $(x,y) \in [0, 2 \pi) \times [0, 2 \pi)$ with periodic boundary conditions, we solve the anelastic equations \eqref{eq:momentum}, \eqref{eq:continuity}, with initial data,
	\begin{linenomath*}\begin{equation*}
			u_0(x,y) = \begin{cases}
				\tanh\left[ \frac{1}{\rho} \left( y - \frac{\pi}{2} \right) \right] &\mbox{if } y < \pi \\
				\tanh\left[ \frac{1}{\rho} \left( \frac{3 \pi}{2} - y \right) \right] &\mbox{if } y \ge \pi \\
			\end{cases}, \hspace{1cm} v_0(x,y) = \delta \sin(x)
	\end{equation*}\end{linenomath*}
	Following \cite{Ghosh2012}, we set $\rho = \frac{\pi}{15}$, $\delta= 0.05$, and approximate the problem on a grid of $256 \times 256$ points, using an underlying WENO5 scheme and a SSP-RK3 time marching method. We take a CFL number of 0.2. For the WENO scheme proposed here (Algorithm \ref{algo:everything}), reconstructions are performed with an ENO5 method. Figure \ref{fig:compare} shows a numerical approximation of the vorticity ($-\frac{\partial u}{\partial y} + \frac{\partial v}{\partial x}$) at the cell centers, at three different times. There are very minor differences between the two WENO schemes, with the scheme proposed here adding slightly less dissipation. Both schemes approximate the solution very well, even at late times. This is in marked contrast to the behavior of well-known central schemes such as the schemes of \cite{Morinishi1998} and \cite{Skamarock2008WRF}. This contrast can be seen from Figure \ref{fig:centralshear} where the vorticity, computed with sixth-order versions of both central schemes are shown. In contrast to the WENO schemes, the central schemes seem to become unstable at later times. This example clearly demonstrates the superiority of WENO type schemes in long time integration of anelastic equations. 
	
	\begin{figure}
		\begin{subfigure}{.48\textwidth}
			\centering
			\includegraphics[width=.8\linewidth]{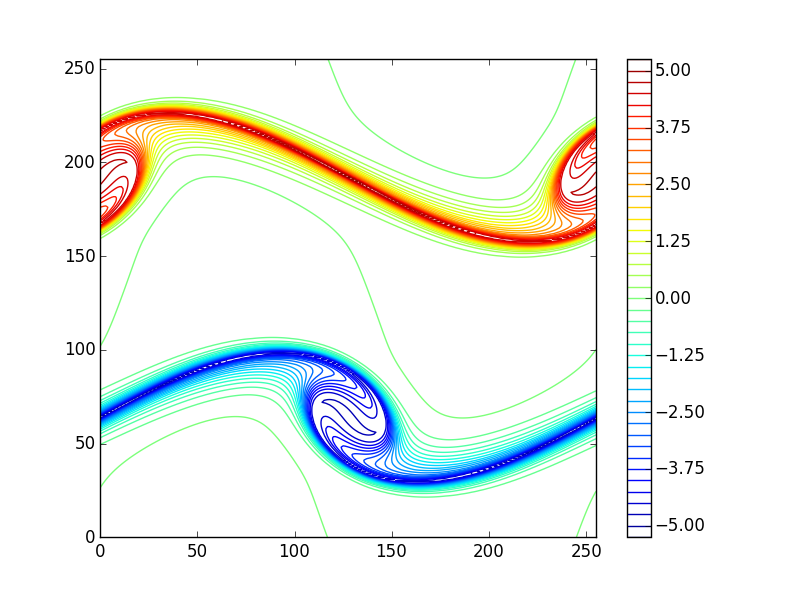}
			\caption{Vorticity at $t=5$, Algorithm \ref{algo:pycles}}
		\end{subfigure}%
		\begin{subfigure}{.48\textwidth}
			\centering
			\includegraphics[width=.8\linewidth]{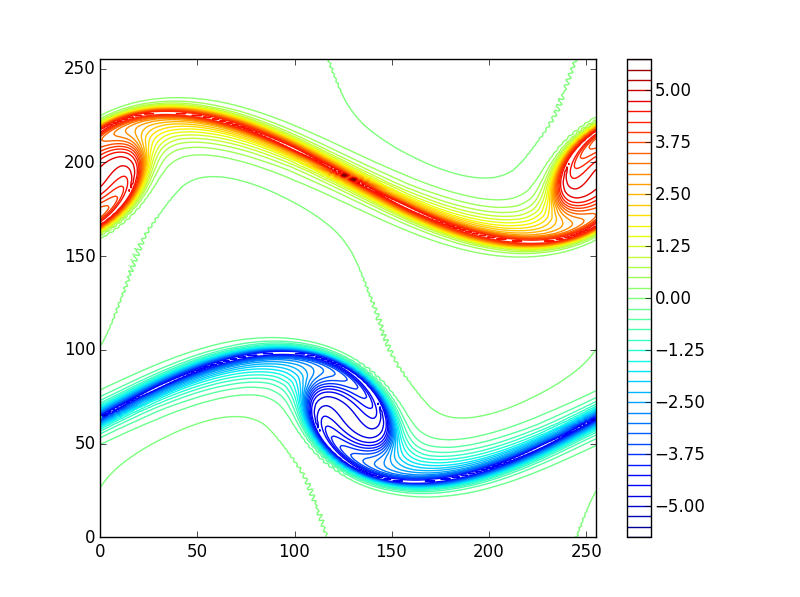}
			\caption{Vorticity at $t=5$, Algorithm \ref{algo:full}}
		\end{subfigure}
		
		\begin{subfigure}{.48\textwidth}
			\centering
			\includegraphics[width=.8\linewidth]{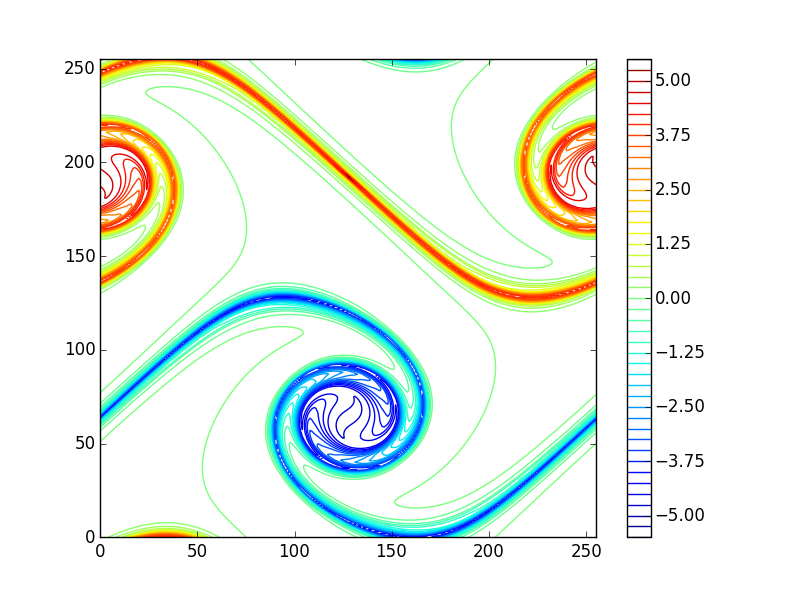}
			\caption{Vorticity at $t=7$, Algorithm \ref{algo:pycles}}
		\end{subfigure}%
		\begin{subfigure}{.48\textwidth}
			\centering
			\includegraphics[width=.8\linewidth]{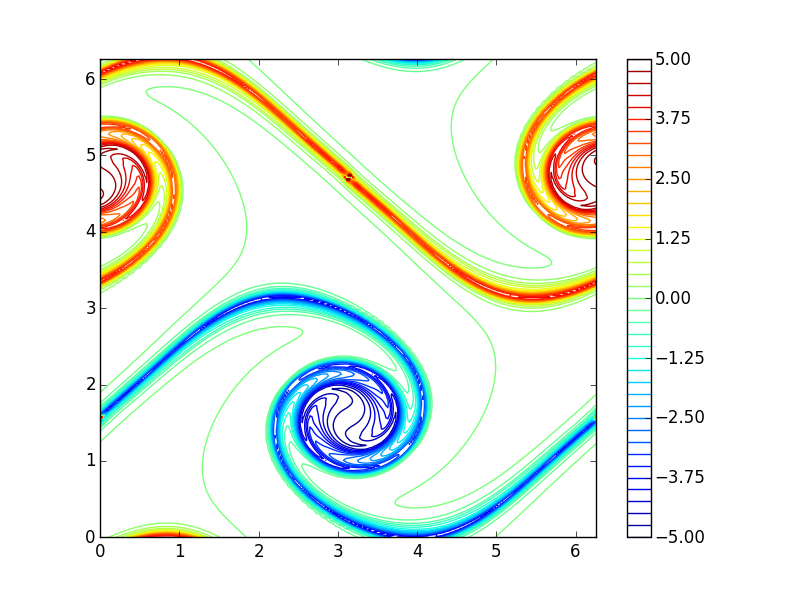}
			\caption{Vorticity at $t=7$, Algorithm \ref{algo:full}}
		\end{subfigure}
		
		\begin{subfigure}{.48\textwidth}
			\centering
			\includegraphics[width=.8\linewidth]{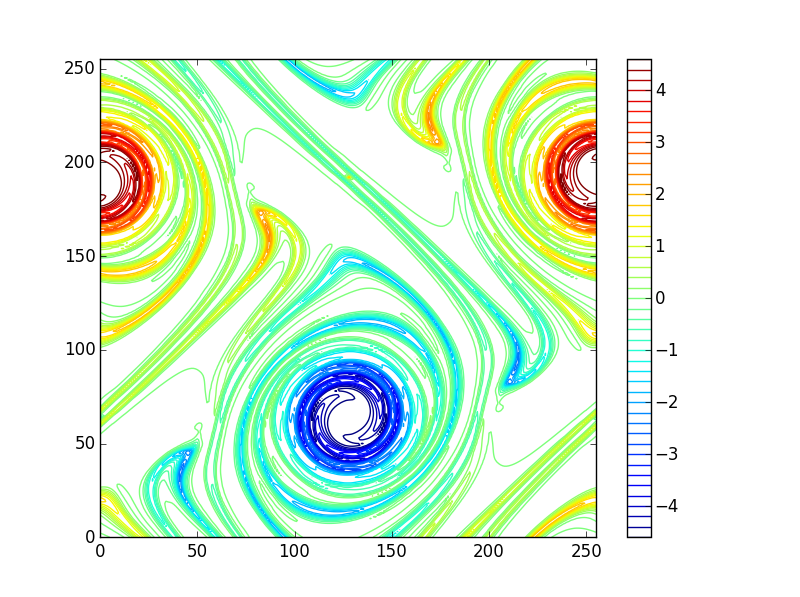}
			\caption{Vorticity at $t=14$, Algorithm \ref{algo:pycles}}
		\end{subfigure}
		\begin{subfigure}{.48\textwidth}
			\centering
			\includegraphics[width=.8\linewidth]{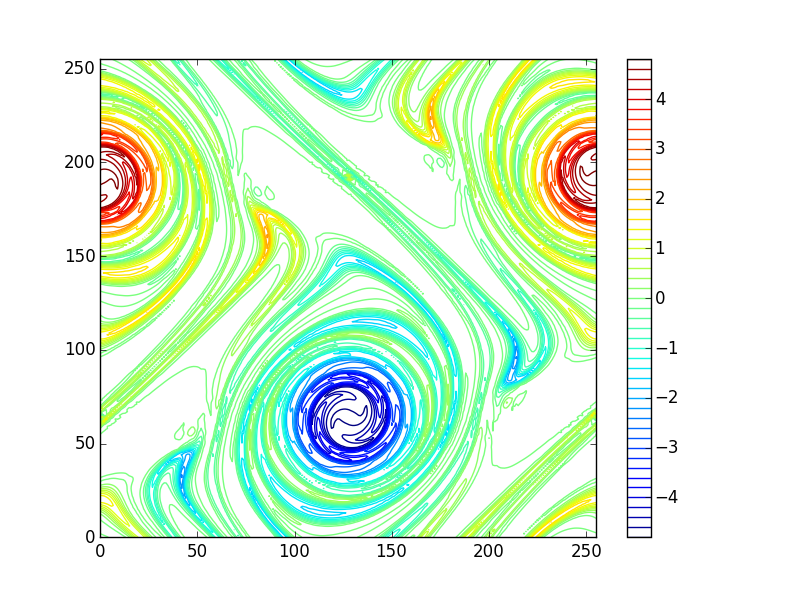}
			\caption{Vorticity at $t=14$, Algorithm \ref{algo:full}}
		\end{subfigure}
		\caption{Vorticity for problem \ref{shear}. We compare Algorithm \ref{algo:pycles} and \ref{algo:full} for the advection of momentum, otherwise using the same pressure solver, time-stepping, etc. }
		\label{fig:compare}
	\end{figure}

	\begin{figure}
		\begin{subfigure}{.48\textwidth}
			\centering
			\includegraphics[width=.8\linewidth]{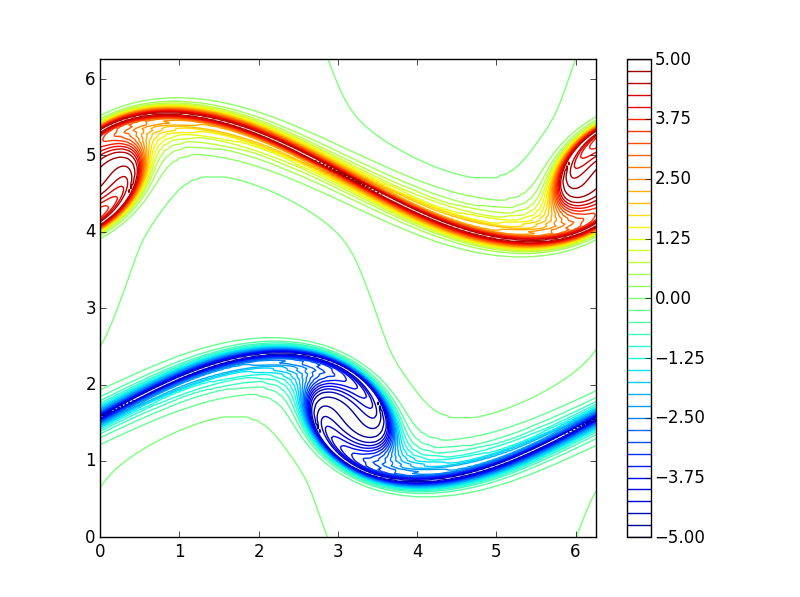}
			\caption{Vorticity at $t=5$, Morinishi-6}
		\end{subfigure}%
		\begin{subfigure}{.48\textwidth}
			\centering
			\includegraphics[width=.8\linewidth]{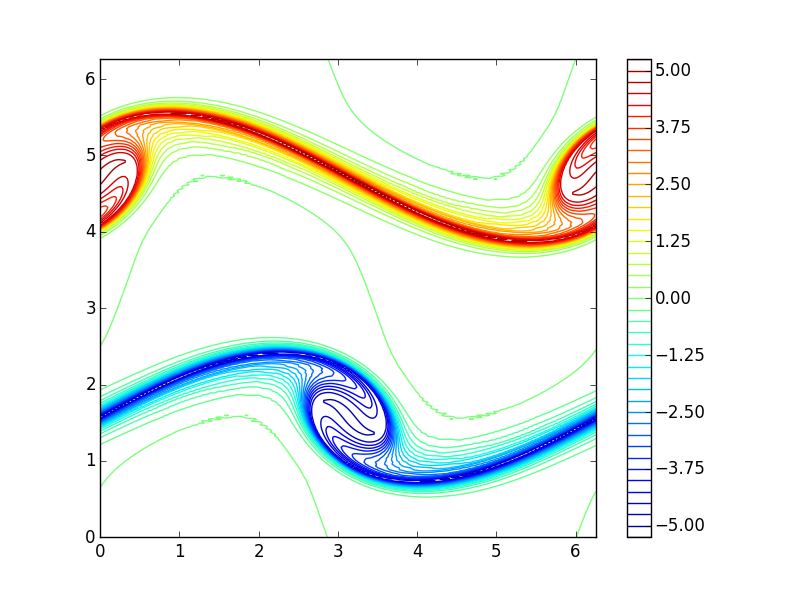}
			\caption{Vorticity at $t=5$, Wicker-Skamarock-6}
		\end{subfigure}
		
		\begin{subfigure}{.48\textwidth}
			\centering
			\includegraphics[width=.8\linewidth]{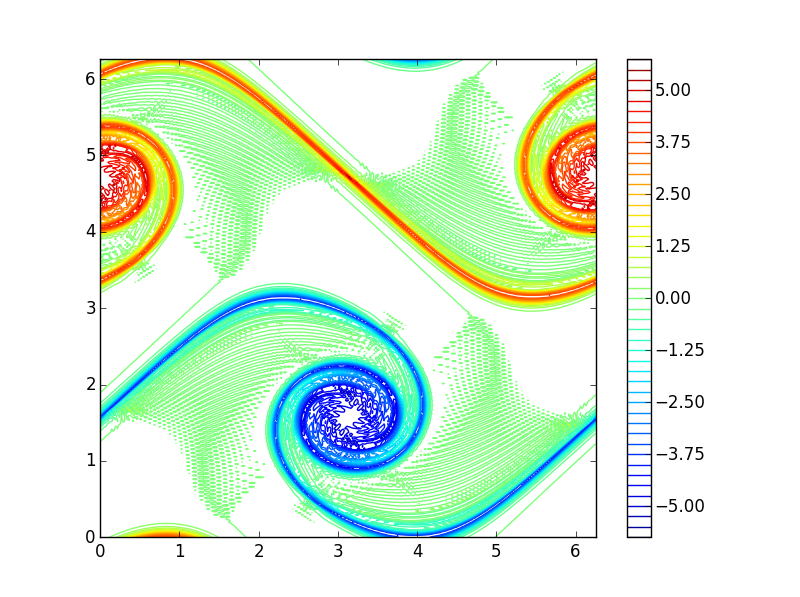}
			\caption{Vorticity at $t=7$, Morinishi-6}
		\end{subfigure}%
		\begin{subfigure}{.48\textwidth}
			\centering
			\includegraphics[width=.8\linewidth]{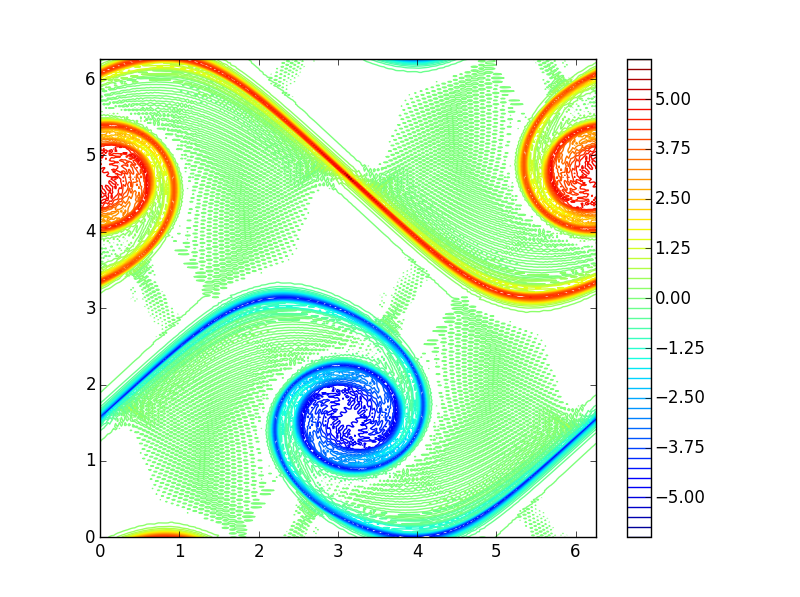}
			\caption{Vorticity at $t=7$, Wicker-Skamarock-6}
		\end{subfigure}
		
		\begin{subfigure}{.48\textwidth}
			\centering
			\includegraphics[width=.8\linewidth]{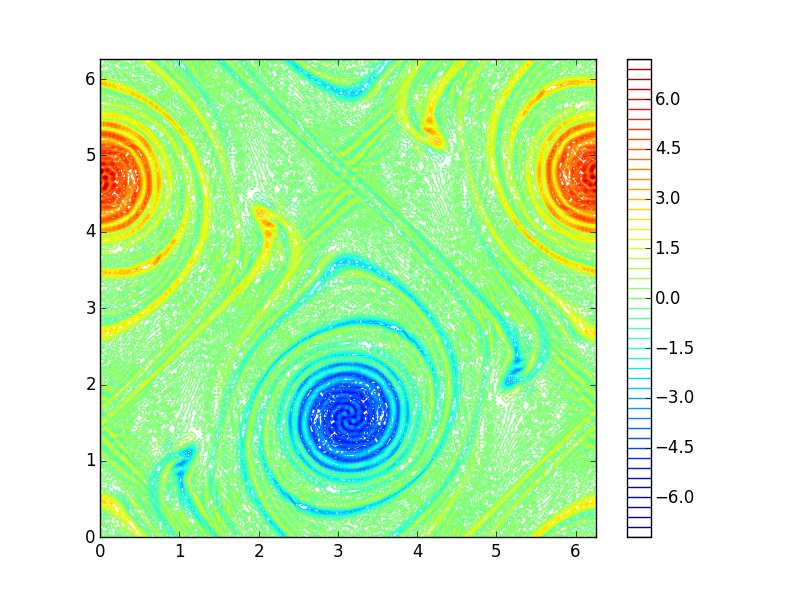}
			\caption{Vorticity at $t=14$, Morinishi-6}
		\end{subfigure}%
		\begin{subfigure}{.48\textwidth}
			\centering
			\includegraphics[width=.8\linewidth]{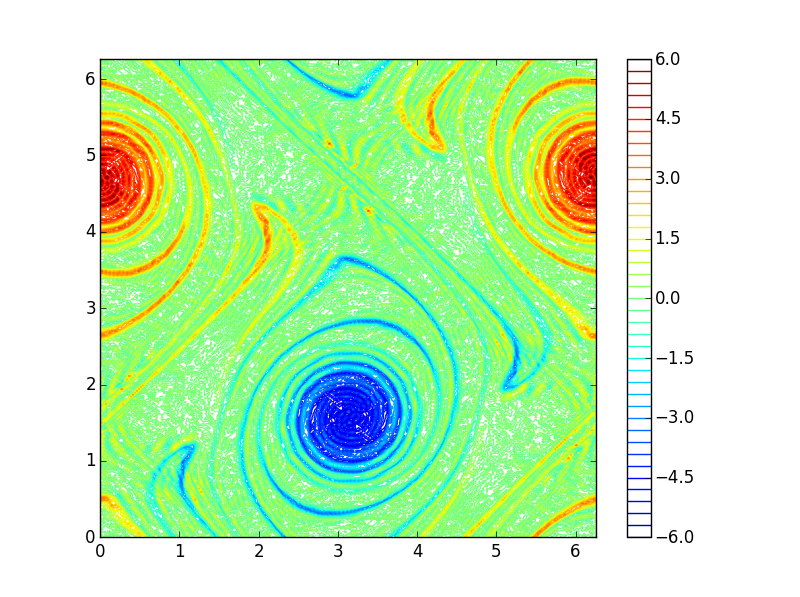}
			\caption{Vorticity at $t=14$, Wicker-Skamarock-6}
		\end{subfigure}
		\caption{Comparison of the central schemes in Appendix \ref{appendixcentral}: 6th order Morinishi and Wicker-Skamarock, at the same time instants as in Fig. \ref{fig:compare}.}
		\label{fig:centralshear}
	\end{figure}
	
	\subsection{Dirichlet boundary conditions}
	\label{dirichlet}
	\revb{
	The examples presented above have showcased the behaviour of our proposed scheme in a domain with periodic boundary conditions. In this section, we demonstrate the fact that the scheme can also be used with other types of boundary conditions; in particular, we present here an example with Dirichlet boundary conditions.
	
	In this example, we \textit{manufacture} the following divergence-free field:
	\begin{align}
	    u(x,y,t) &= (1+t) \sin(x) \cos(y) \label{eq:manufsolu} \\ 
	    v(x,y,t) &= -(1+t) \cos(x) \sin(y) \label{eq:manufsolv}.
	\end{align}
	and we add source terms $\Phi_u$, $\Phi_v$ to eq. \eqref{eq:momentum} so that the field $(u,v)$ verifies the equations
	\begin{align}
	    \partial_t u + \partial_x (u^2) + \partial_y (uv) + \partial_x p &= \Phi_u \label{eq:manufu} \\ 
	    \partial_t v + \partial_x (uv) + \partial_y (v^2) + \partial_y p &= \Phi_v. \label{eq:manufv}
	\end{align}
	Specifically, this implies a choice of
	\begin{align} \Phi_u(x,y,t) &= \sin(x)\cos(y) + \frac{(1+t)^2}{2} \sin(2x), \\
	\Phi_v(x,y,t) &= -\cos(x)\sin(y) + \frac{(1+t)^2}{2} \sin(2y). \end{align}
	
	In other words, consider the domain $\Omega = [0, \pi]^2$ and seek a solution to  \eqref{eq:manufu}-\eqref{eq:manufv}, subject to the boundary conditions \eqref{eq:manufsolu}-\eqref{eq:manufsolv} for $(x,y) \in \partial \Omega$, for all $t \in [0, 0.1]$.
	
	The implementation of this configuration is straightforward: the exact values of the solution \eqref{eq:manufsolu}-\eqref{eq:manufsolv} can be used to fill as many layers of ghost cells as are needed for the stencil of the high-order method to be well-defined for all cells.
	
	We fix a small time-step of $\Delta t = 10 ^{-5}$ to make the spatial discretization the leading error term, at least for coarse grids; and we use Algorithm \ref{algo:everything} to approximate the solution. The errors obtained with Algorithm \ref{algo:everything} using different orders of accuracy are shown in Table \ref{tab:dirichlet}. The discrete $L^1$ error for $u$ at time $t=0.1$ is shown; the behaviour of the error for $v$ is completely analogous. As seen from this table, the high-order WENO schemes are able to approximate this solution with the designed order of accuracy, even up to the boundary. In particular, the 5th-order scheme has an order of accuracy of approximately $5$ and the seventh-order scheme also has the design accuracy on coarse grids although the accuracy does deteriorate on fine grids due to the fact that the errors are very very low on such grids. 
	
	Thus, this experiment brings forth the ability of this method to approximate different boundary conditions. However, we would like to remark that in general, it may not be possible to construct WENO schemes with very high-order up to the boundary and special boundary treatments, such as SBP-SAT schemes might be necessary.  Nevertheless, in \cite{Pressel2017}, the authors find that a similar scheme applied to the simulation of atmospheric flows produces more realistic results overall than lower-order central schemes, even when paired with a low-order approximation to no-flow boundary conditions.}
	
	\begin{table}
		\begin{center}
			\begin{tabular}{c || c c | c c | c c  }
				$n_x$ & u (3rd) & EOC & u (5th) & EOC  & u (7th) & EOC \\ \hline
				16 &  0.0179 & -- &  0.000505 & -- & 2.60e-5 & --   \\
				32 &  0.00551 & 1.70 & 2.07e-5 & 4.61 & 3.79e-7 &  6.10  \\
				64 & 0.00137  & 2.01 &  5.29e-7 & 5.29 & 1.22e-8 & 4.96  \\
				128 & 0.000320 &  2.10 & 1.91e-8 & 4.79 & 5.41e-9 & 1.17  
			\end{tabular}
		\end{center}
		\caption{Discrete $L^1$ numerical error for experiment \ref{dirichlet} (manufactured solution with Dirichlet BCs)}
		\label{tab:dirichlet}
	\end{table}

	\section{An arbitrarily high-order WENO scheme for advecting scalars in anelastic flows}
	\label{sec:5}
	Advection of scalars by an anelastic flow is modeled as:
	\begin{linenomath*}\begin{equation}\label{eq:scalar}
			\frac{\partial \phi}{\partial t} + \frac{1}{\rho_0} \sum_{i=1}^{d} \frac{\partial(\rho_0 u_i \phi)}{\partial x_i} = \Phi,
	\end{equation}\end{linenomath*}
	
	cf. eq. \eqref{eq:entropy}. Here, $\phi: \mathbb{R}^d \to \mathbb{R}$ is the scalar that is advected by the flow field $U = \{u_i\}_{1\leq i \leq d}$, which solves the anelastic equations \eqref{eq:momentum}, \eqref{eq:continuity}, and $\Phi$ represents sources and sinks for the scalar.		
	Our aim is to discretize the scalar advection equation \eqref{eq:scalar} on the Arakawa-C staggered grid. As pointed out in section \ref{sec:2} and shown in Figure \ref{fig:scalars}, point values of the scalar $\phi$ are stored at cell centers, i.e $\phi_{i+\frac{1}{2},j+\frac{1}{2}} \approx \phi(x_{i+\frac{1}{2}},y_{j+\frac{1}{2}})$. 
	
	In \cite{Pressel2015}, a WENO scheme for the advection of scalars is presented, with spatial derivatives being discretized in a manner completely analogous to Algorithm \ref{algo:pycles}. Reiterating the argument of Theorem \ref{theorem}, one can readily show that the resulting scheme is at most second-order accurate, even if very high-order interpolating polynomials are utilized within the WENO procedure. Our aim is to design an arbitrarily high-order accurate WENO scheme. 
	\begin{figure}
		\begin{center}
			\begin{tikzpicture}
			
			\draw[step=2cm,gray,very thin] (-0.5, -0.5) grid (6.5, 2.5);
			\draw[thick, line width=0.7mm] (2,0) rectangle (4,2);
			\node[rectangle,draw,xscale=1.25, yscale=1.25,rotate=45] at (3,1) {};
			\draw (2,1) circle[radius = 5pt];
			\draw (4,1) circle[radius = 5pt];
			\fill (3,1) circle[radius = 3pt];
			\fill (1,1) circle[radius = 3pt];
			\fill (5,1) circle[radius = 3pt];
			\node at (4.8, 1.4) {$F^{\phi,x}_{i+1,j+\frac{1}{2}}$};
			\node at (2.9, 1.4) {$\phi_{i+\frac{1}{2}, j+\frac{1}{2}}$};
			\node at (3, 0.6) {$\tilde{u}_{i+\frac{1}{2}, j+\frac{1}{2}}$};
			\end{tikzpicture}
		\end{center}
		\caption{Reconstruction of $\frac{\partial}{\partial x}(u \phi)$ on a staggered grid. The diamond marks the location of $\phi$ in the highlighted cell; numerical fluxes must be calculated at the circles. Interpolations of $u$ are required at the points marked with dots.}
		\label{fig:scalars}
	\end{figure}
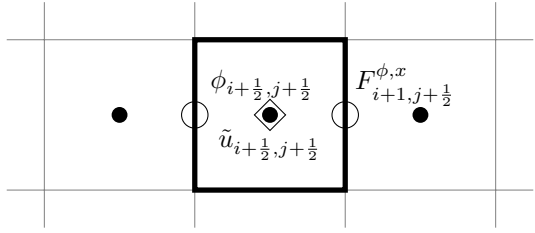
	\subsection{The high-order scheme}
	\label{scalarshigh}
	We start by using the divergence constraint and recasting the scalar advection equation \eqref{eq:scalar} in the \emph{non-conservative form},
	\begin{linenomath*}\begin{equation}\label{eq:scalar2}
			\frac{\partial \phi}{\partial t} + \frac{1}{\rho_0} \sum_{i=1}^{3} \frac{\partial(\rho_0 u_i \phi)}{\partial x_i} = \Phi
			\qquad \stackrel{\nabla \cdot (\rho_0 U) = 0}
			{\iff} \qquad
			\frac{\partial \phi}{\partial t} + \sum_{i=1}^{3}  u_i \frac{\partial \phi}{\partial x_i} = \Phi
	\end{equation}\end{linenomath*}
	We choose to discretize this non-conservative form of scalar advection on account of observed better stability of this formulation, \rev{due to the issue outlined in Remark \ref{rem:discrdiv}; as well as for slightly lower computational costs. We observe a behaviour with this formulation which is close to conservative, as numerical example \ref{bubble} shows.}
	
	The basis of our WENO discretization of \eqref{eq:scalar2} is to adapt ideas presented in section \ref{sec:3} to the present setting. For simplicity, we consider the problem in two space dimensions and assume a constant density $\rho_0$. Then, \eqref{eq:scalar2} reduces to
	\begin{equation}
		\label{eq:scal3}
		\phi_t + u \partial_x \phi + v \partial_y \phi = \Phi.
	\end{equation}
	Here, $U = (u,v)$ is the velocity that solves \eqref{eq:momentum}, \eqref{eq:continuity} in two space dimensions. A semi-discrete finite difference scheme (on the Arakawa-C grid) for discretizing \eqref{eq:scal3} is given by
	\begin{equation}
		\label{eq:ss1}
		\frac{d}{dt} \phi_{i+\frac{1}{2},j+\frac{1}{2}} (t) + D^{\phi,x}_{i+\frac{1}{2}, j+\frac{1}{2}} + D^{\phi,y}_{i+\frac{1}{2}, j+\frac{1}{2}} = \Phi_{i+\frac{1}{2},j+\frac{1}{2}}.
	\end{equation} 
	Here, $\Phi_{i+\frac{1}{2},j+\frac{1}{2}} = \Phi(x_{i+\frac{1}{2}},y_{j+\frac{1}{2}})$ is a point evaluation of the source term in \eqref{eq:scal3} and 
	\begin{equation}
		\label{eq:ss2}
		D^{\phi,x}_{i+\frac{1}{2}, j+\frac{1}{2}} \approx (u \partial_x \phi)(x_{i+\frac{1}{2}},y_{j+\frac{1}{2}}), \quad D^{\phi,y}_{i+\frac{1}{2}, j+\frac{1}{2}} \approx (v \partial_y \phi)(x_{i+\frac{1}{2}},y_{j+\frac{1}{2}}),
	\end{equation}
	are discretizations of the non-conservative products of velocities and spatial derivatives of the scalar. Thus, the scheme \eqref{eq:ss1} needs a routine to calculate spatial derivatives in a non-oscillatory manner and to very high order of accuracy. We follow ideas presented in section \ref{sec:3} and propose the following algorithm for defining $ D^{\phi,x}_{i+\frac{1}{2}, j+\frac{1}{2}}$. The definition of $D^{\phi,y}_{i+\frac{1}{2}, j+\frac{1}{2}}$ is analogous. 
	\begin{algorithm}
		\label{algo:scalars}~
		\textbf{Goal:} Find $D^{\phi,x}_{i+\frac{1}{2}, j+\frac{1}{2}}$ such that:
		
		\begin{linenomath*}\begin{equation*}
				D^{\phi,x}_{i+\frac{1}{2}, j+\frac{1}{2}} = \left( u  \frac{ \partial \phi}{\partial x} \right) (x_{i+\frac{1}{2}}, y_{j+\frac{1}{2}}) + O(\Delta x^{2k-1} + \Delta y^{2k-1}), \quad 1\leq k \in {\mathbb N}
		\end{equation*}\end{linenomath*}
		
		\begin{enumerate}
			\item Compute high-order approximations to the velocity field $u$ at cell centers,
			\[ \tilde{u}_{i+\frac{1}{2},j+\frac{1}{2}} = u(x_{i+\frac{1}{2}}, y_{j+\frac{1}{2}}) + O(\Delta x^{2k-1}) \]

			\item Compute the two biased WENO-$(2k-1)$ reconstructions of $\phi$ along $x$ at $(x_{i+1}, y_{j+\frac{1}{2}})$:
			\begin{align*} 
				\hat{\phi}_{i+1, j+\frac{1}{2}}^- & = W_{+\frac{1}{2}}( \phi_{i+\frac{1}{2}-(k-1), j+\frac{1}{2}}, \phi_{i+\frac{1}{2}-(k-2), j+\frac{1}{2}}, \dots, \phi_{i+\frac{1}{2}+(k-1), j+\frac{1}{2}} ) \\
				\hat{\phi}_{i+1, j+\frac{1}{2}}^+ & = W_{-\frac{1}{2}}( \phi_{i+\frac{1}{2}-(k-2),j+\frac{1}{2}}, \phi_{i+\frac{1}{2}-(k-1),j+\frac{1}{2}}, \dots, \phi_{i+\frac{1}{2}+k,j+\frac{1}{2}} ) 
			\end{align*}
			
			\item Choose a numerical flux function $G(u^-, u^+)$ (for instance upwind with respect to $\tilde{u}_{i+\frac{1}{2},j+\frac{1}{2}}$), and set
			\[ \hat{\phi}_{i+1, j+\frac{1}{2}} := G(\hat{\phi}_{i+1, j+\frac{1}{2}}^-, \hat{\phi}_{i+1, j+\frac{1}{2}}^+) \]
			
			\item Define
			\[ D^{\phi,x}_{i+\frac{1}{2}, j+\frac{1}{2}} = \tilde{u}_{i+\frac{1}{2},j+\frac{1}{2}} \frac{ \hat{\phi}_{i+1, j+\frac{1}{2}} -  \hat{\phi}_{i, j+\frac{1}{2}}}{\Delta x} \]
		\end{enumerate}
	\end{algorithm}
	
	\begin{remark}
		For the first step of Algorithm \ref{algo:scalars} we require high-order interpolations of the advecting velocity at cell centers. We recall from section \ref{sec:3} that these values have been already computed as an intermediate step of the interpolation procedure (Algorithm \ref{algo:2dinterp}, step \ref{step:intermediaterec}). 
	\end{remark}
	We can define $D^{\phi,y}_{i+\frac{1}{2}, j+\frac{1}{2}}$ by readily adapting Algorithm \ref{algo:scalars} and using the reconstructed values of $v$ at cell centers. The scheme \eqref{eq:ss1} can be easily extended to include variable densities and to discretize \eqref{eq:scalar2} in three space dimensions. A straightforward truncation error analysis, together with the arbitrarily high-order nature of WENO interpolations, demonstrates that the resulting scheme is as accurate as the order of the underlying WENO reconstructions.
	
	\begin{remark}
	\rev{
	WENO-based algorithms ensure positivity, and in general pointwise bounds on the advected scalar, better than the central scheems, as they control spurious oscillations near sharp gradients. Note however that we provide no rigorous guarantees of positivity preservation for scalars, let alone maximum or minimum principles. One can ensure bound preservation through the use of limiters, as e.g. \cite{Liang2014}, which would however reduce the order of accuracy; or heuristics such as the ingenious idea of ``sweeping'' in \cite{Liu2017}. To the best knowledge of the authors, no rigorously high-order, finite difference scheme exists which preserves positivity. For cases where high-order positivity is required, finite volume approaches are better suited. Note however that these require the evaluation of interpolations at many quadrature points, and therefore are computationally expensive.
	}
	\end{remark}
	
	\subsection{Numerical experiments}
	\subsubsection{Passive scalar}
	\label{passivescalar}
	We design this experiment to test the order of accuracy of the WENO schemes. In order do so, we will \emph{manufacture} an exact solution of \eqref{eq:scal3} in the two-dimensional periodic domain $[0,2\pi]^2$ by setting,
	\begin{linenomath*}\begin{align*} \label{eq:passivescalar}
			u(x,y,t) &= -\cos(t)\sin(x)\sin(2y) \\
			v(x,y,t) &= \cos(t)\cos(x) \sin(y)^2 \\
			\phi(x,y,t) &= 2 + \cos(x)\sin(y)\cos(t),
	\end{align*}\end{linenomath*}
	and computing the source term $\Phi$ analytically in order to satisfy \eqref{eq:scal3}, as well as source terms for velocity so that eq. \eqref{eq:momentum} holds. We observe that the velocity field $(u,v)$ defined above is divergence free and satisfies the two-dimensional form of the anelastic equations with a
	suitable pressure term. 
	
	We compare the approximations computed for $T=1$ with the exact solution, measuring error in $L^1$ norm. We choose $\Delta t$ small enough to ensure that the leading error term comes from spatial discretization. Table \ref{tab:passivescalar} shows the numerical error for $u$ (as the values for $v$ are very similar) and $\phi$. In this table, we compare the errors with the WENO scheme of \cite{Pressel2015} (denoted by Algorithm \ref{algo:pycles}) and the WENO scheme proposed here (Algorithm \ref{algo:everything}-\ref{algo:scalars}). We see from the table that that there is a loss of order of convergence with the previously existing WENO scheme and it is limited to at most second order accuracy even when interpolation polynomials of a higher degree are used. On the other hand, the proposed scheme recovers the design order of accuracy and has an error that is three to four orders of magnitude smaller than the existing WENO schemes. 
	
	\begin{table}
		\begin{center}
			\textbf{$L^1$ error (velocity and scalar)}
			\begin{tabular}{c || c c | c c || c c | c c }
				& \multicolumn{4}{c}{(Alg. \ref{algo:pycles})} & \multicolumn{4}{c}{(Alg. \ref{algo:everything}-\ref{algo:scalars})}\\ \hline
				$n_x$ & u (5th) & EOC & $\phi$ (5th) & EOC  & u (5th) & EOC & $\phi$ (5th) & EOC \\ \hline
				16 & 0.103 & -- & 0.395 & -- & 0.127 & -- & 0.563 & --  \\
				32 & 0.0367 & 1.489 & 0.126 & 1.651 & 0.01 & 3.661 & 0.0497 & 3.503  \\
				64 & 0.0104 & 1.823 & 0.0341 & 1.882  & 0.00042 & 4.578 & 0.00198 & 4.651 \\
				128 & 0.00269 & 1.947 & 0.00895 & 1.931 & 1.47e-05 & 4.84 & 7.68e-05 & 4.685 \\
				256 & 0.000688 & 1.968 & 0.00231 & 1.957  & 4.85e-07 & 4.916 & 2.62e-06 & 4.872 \\
				512 & 0.000174 & 1.983 & 0.000586 & 1.977 & 1.57e-08 & 4.949 & 8.54e-08 & 4.942
			\end{tabular}
		\end{center}
		\caption{Discrete $L^1$ numerical error for experiment \ref{passivescalar} (Passive scalar)}
		\label{tab:passivescalar}
	\end{table}
	
	\subsubsection{Active scalar: entropy and buoyancy}
	\label{bubble}
	Next, we consider a more realistic example of an \emph{active scalar}, transported by the anelastic flow. The scalar represents $s$, the dry entropy. This, in turn, induces buoyancy (cf. the equations for momentum \eqref{eq:momentum} and entropy \eqref{eq:entropy}), which will drive the evolution, following the equations:
	
	\begin{linenomath*}\begin{equation} \label{eq:buoy}
			b(x,y,z) = g \frac{T(x,y,z) - T_0(z)}{T_0(z)} 
	\end{equation} \end{linenomath*}
	\begin{linenomath*}\begin{equation}
			T(x,y,z) = \exp\left( \frac{ s(x,y,z) / \rho_0(z) - \tilde{s} + R_d\log(p_0(z)/\tilde{p})}{c_{pd}}\right), \label{eq:stoT}
	\end{equation} \end{linenomath*}
	where $T$ is air temperature, variables with subindex 0 stand for values in the reference state, tildes denote standard surface values ($\tilde{p} = 10^5\text{ Pa}$, $\tilde{s} = 6864.8 \text{ J kg}^{-1}\text{K}^{-1}$), $g=9.8\ m/s^2$, $R_d = 287.1\text{ J kg}^{-1} \text{K}^{-1}$ and $c_{pd} = 1004\text{ J kg}^{-1} \text{K}^{-1}$. A detailed derivation of expressions \eqref{eq:buoy} and \eqref{eq:stoT} can be found in \cite{Pressel2015}, for the more general case of moist air.
	
	The test case is taken from Straka et al. \cite{Straka1993}. The domain is $D = [-25600, 25600] \times [0, 6400]$. Here, the boundaries at $z=\{0, 6400\}$ are no-flow, with the boundaries in $x$ still periodic. We begin by a cold perturbation of the reference state. Let $L: \mathbb{R}^2 \longrightarrow \mathbb{R}^+$, 
	\begin{linenomath*}\begin{equation*}
			L(x,z) = \left( \left(\frac{x}{4000}\right)^2 + \left(\frac{z - 3000}{2000}\right)^2 \right)^{\frac{1}{2}}
	\end{equation*} \end{linenomath*}
	The perturbation in temperature from the reference atmospheric state is \begin{linenomath*}\begin{equation*}
			\Delta T(x,z) = -7.5\ (\cos( \min \{L(x,z), 1\}\  \pi) +1 )
	\end{equation*} \end{linenomath*}

	The evolution of the bubble is computed over a $512 \times 512$ grid, and results, computed with the fifth order versions of Algorithm \ref{algo:full} for momentum and Algorithm \ref{algo:scalars} for scalars, are presented in Figure \ref{fig:bubbleevol}. Due to the symmetry of the problem, plots show only the subdomain $[0, 22400]\times[0, 4000]$ (comprising $192 \times 320$ mesh points). The plot shows values for entropy temperature,
	\[ \theta(x,z) := 300  \exp\left(\frac{s(x,z) - \tilde{s}}{c_{pd}}\right). \]
	
	Variations in the total amount of entropy, during the full simulation, remain in the order of $0.01\%$. This suggests that the non-conservative character of the scheme has a minor effect, if at all, on the results. Moreover, the figure also shows that the proposed high-order WENO scheme of this paper approximates this realistic bubble evolution very well.

	\begin{figure}
		\begin{subfigure}{.5\textwidth}
			\centering
			\includegraphics[width=.8\linewidth]{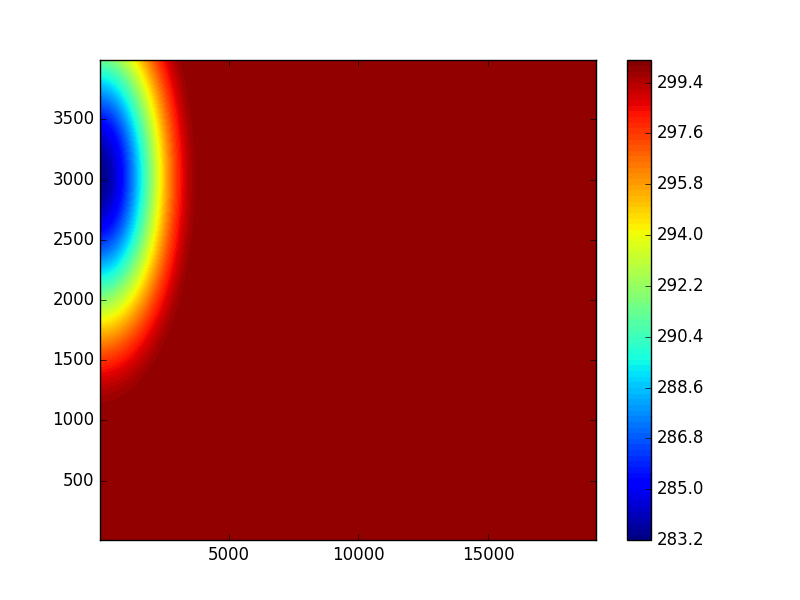}
			\caption{$T=0$}
		\end{subfigure}%
		\begin{subfigure}{.5\textwidth}
			\centering
			\includegraphics[width=.8\linewidth]{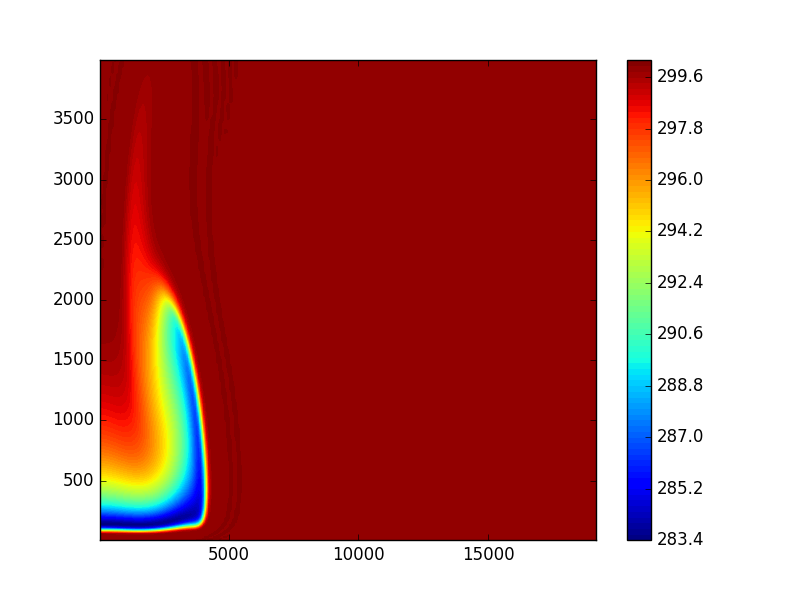}
			\caption{$T=300$}
		\end{subfigure}
		
		\begin{subfigure}{.5\textwidth}
			\centering
			\includegraphics[width=.8\linewidth]{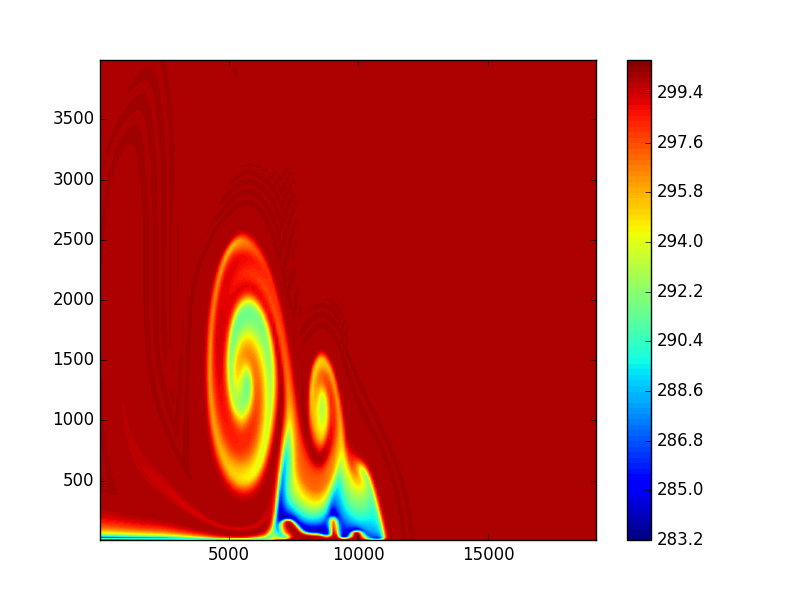}
			\caption{$T=600$}
		\end{subfigure}%
		\begin{subfigure}{.5\textwidth}
			\centering
			\includegraphics[width=.8\linewidth]{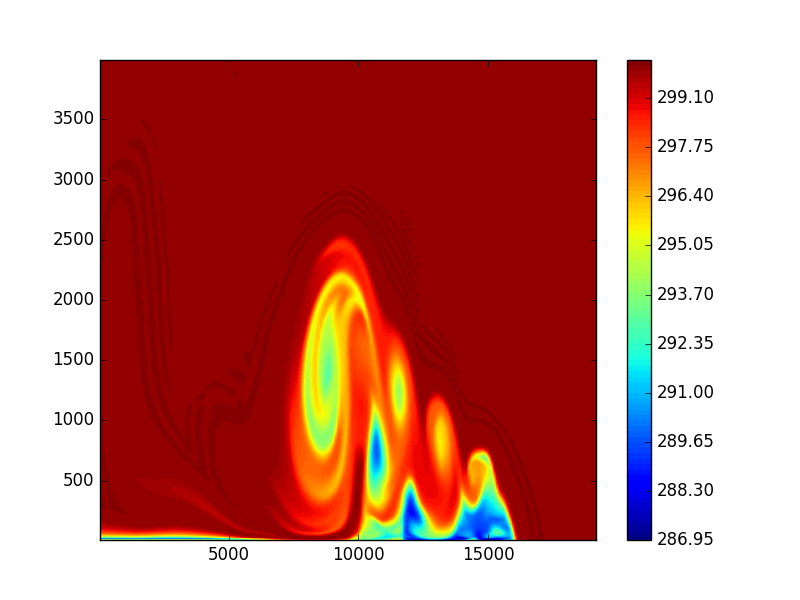}
			\caption{$T=900$}
		\end{subfigure}
		\caption{Entropy temperature $\theta$ for test case \ref{bubble} in a 2D vertical domain ($y$ axis represents height)}
		\label{fig:bubbleevol}
	\end{figure}

	\section{Summary}
	We consider the anelastic equations \eqref{eq:momentum}-\eqref{eq:continuity} and passive or active scalars, advected by the anelastic flow \eqref{eq:scalar}. These equations arise
	frequently in atmospheric sciences and other related fields. Although (high-order) central finite difference schemes have been extensively used to discretize these equations, they are deficient at
	resolving sharp gradients, such as in large eddy simulations (LES) of clouds, or at ensuring positivity or bound preservation for scalars, \rev{mainly due to spurious oscillations}. Consequently, WENO type schemes, that are well-suited for
	simulation of flows with sharp gradients, have been designed in recent years to approximate the anelastic flow equations. 
	
	We show that these WENO schemes on the underlying staggered grid, such as the scheme described in \cite{Pressel2015}, suffer from a loss in the design order of convergence due to the staggering of the 
	velocity components on cell edges. In particular, the scheme is restricted to at most second order of accuracy, irrespective of the underlying (high-order) piecewise polynomial reconstructions. Given this deficiency, we propose a novel scheme that is guaranteed to be arbitrarily high-order accurate. This scheme is based on a novel ENO interpolation of the velocity components to a set of collocated points and the subsequent use of a WENO reconstruction of
	the spatial derivatives. This combination ensures that the overall scheme has a order of accuracy which is consistent with the underlying WENO method. The advection scheme is complemented with standard pressure solvers and SSP Runge-Kutta methods. 
	Moreover, we design a WENO scheme that discretizes scalars advected by the anelastic flow \eqref{eq:scalar}. This scheme is based on a non-conservative formulation and reuses the ENO interpolated velocities at cell centers. It is also shown to be arbitrarily high-order accurate. 
	
	We have presented various numerical experiments comparing different schemes and summarize the results below:
	\begin{itemize}
		\item Both the proposed WENO scheme and the WENO scheme of \cite{Pressel2015} are very robust at resolving discontinuities and sharp gradients. We noticed a slight increase in robustness with our proposed scheme. This robustness of WENO schemes
		is in marked contrast to the failure of central difference schemes at resolving sharp gradients on account of the generation of spurious oscillations.
		
		\item The proposed WENO scheme was verified to possess an order of convergence that it is consistent with the design order, i.e the order of accuracy dictated by the underlying order of piecewise polynomials. This should be contrasted with 
		existing WENO scheme of \cite{Pressel2015} and even the central difference schemes of \cite{Skamarock2008WRF}, where the observed order of convergence was at most two. We demonstrate this gain in accuracy with both velocity fields as well as scalars. This guaranteed high-order of accuracy is not just of academic interest. In figure \ref{fig:error}, we show an example where the arbitrarily high-order WENO schemes have the best error vs. cost performance of all the competing schemes.
		
		\item Both WENO schemes were observed to be particularly suitable for problems with long time integration, when compared to central schemes. 
		
	\end{itemize}
	
	Based on the above observations, it is reasonable to argue that WENO schemes are considerably superior when approximating anelastic flows containing sharp gradients. Since such flows are ubiquitous, WENO schemes constitute a natural
	discretization framework for anelastic flows. The main advantage of the proposed scheme over existing WENO schemes is the increased order of accuracy. This results in errors that are several orders of magnitude lower than those generated by existing WENO schemes on staggered grids. On the other hand, the proposed scheme is more computationally expensive as additional ENO interpolations have to be performed. However, the additional cost is of the same order as the cost for the WENO
	reconstructions in the first place. Hence, this additional cost can be justified if higher order of accuracy is desired for the given problem, particularly on reasonably coarse grids. 	
	
	\section*{Acknowledgments}
	The research of SM was partially funded by European Research Council (ERC) Consolidator Grant (CoG) 770880 COMANFLO.
	
	\begin{appendices}
		\section{Central schemes on a staggered grid}
		\label{appendixcentral}
		
		In this appendix we review existing, popular central schemes for the advection of scalars on a staggered grid. 
		
		\subsection{Wicker-Skamarock schemes}
		\label{wickerskamarock}
		The discussion of the Wicker-Skamarock schemes will be done here for 1D grids. To generalize to higher dimensions it suffices to repeat the procedure independently for the flux in every direction.
		
		The scheme discussed in this subsection is presented in Wicker and Skamarock \cite{Wicker2002} as follows: for $\frac{\partial}{\partial t}\phi + \frac{\partial}{\partial x}(u \phi) = 0$, we use a discretization of the type:
		\begin{equation}
			\frac{\phi_{i+\frac{1}{2}}^{n+1} - \phi_{i+\frac{1}{2}}^n}{\Delta t} = - \frac{F_{i+1}^n - F_{i}^n}{\Delta x} \label{eq:wickerskamarock}
		\end{equation}
		where several possibilities (with different orders of accuracy) are proposed for the spatial discretization, e.g.:
		\begin{linenomath*}\begin{equation} \label{eq:f4}
				F^{4th}_{i} :=   \frac{u_{i}}{12}[ 7(\phi_{i+\frac{1}{2}} + \phi_{i-\frac{1}{2}}) - (\phi_{i+\frac{3}{2}} + \phi_{i-\frac{3}{2}}) ]
		\end{equation} \end{linenomath*}
		\begin{linenomath*}\begin{equation} \label{eq:f6}
				F^{6th}_{i} = \frac{u_i}{60} [ 37(\phi_{i+\frac{1}{2}} + \phi_{i-\frac{1}{2}}) - 8(\phi_{i+\frac{3}{2}} + \phi_{i-\frac{3}{2}}) + (\phi_{i+\frac{5}{2}} + \phi_{i-\frac{5}{2}})] 
		\end{equation}\end{linenomath*}

		\subsubsection{Order of convergence for the Wicker-Skamarock scheme}
		
		In \cite{Wicker2002}, the above expressions for the flux are shown to produce high-order accurate results for constant advecting velocities. However, in a general case where velocities are not constant, they are limited to second-order accuracy, as acknowledged in Skamarock and Klemp \cite{Skamarock2008}. As an example, we prove it here for $F^{4th}$.
		
		\begin{theorem}
			\label{skamarock2}
			For $u$, $\phi$ sufficiently smooth, and taking $F = F^{4th}$ as in \eqref{eq:f4}, the central scheme \eqref{eq:wickerskamarock} is at most second order accurate.
		\end{theorem}
		
		\begin{proof}
			Let us write the fluxes as:
			
			\begin{linenomath*}\begin{equation*}
					F^{4th}_{i} = u_i Q_i, \quad F^{4th}_{i+1} = u_{i+1} Q_{i+1}, \qquad \text{ with } Q_{i} := \frac{1}{12} \left( 7(\phi_{i+\frac{1}{2}} + \phi_{i-\frac{1}{2}}) - (\phi_{i+\frac{3}{2}} + \phi_{i-\frac{3}{2}}) \right)
			\end{equation*}\end{linenomath*}
			
			By Taylor expansion, we know that:
			\begin{linenomath*}\begin{equation*}
					u_i = u(x_{i+\frac{1}{2}}) - u'(x_{i+\frac{1}{2}})\frac{\Delta x}{2} + \frac{u''(x_{i+\frac{1}{2}})}{2}\frac{\Delta x^2}{4} + O(\Delta x^3)
			\end{equation*}\end{linenomath*}
			\begin{linenomath*}\begin{equation*}
					u_{i+1} = u(x_{i+\frac{1}{2}}) + u'(x_{i+\frac{1}{2}})\frac{\Delta x}{2} + \frac{u''(x_{i+\frac{1}{2}})}{2}\frac{\Delta x^2}{4} + O(\Delta x^3)
			\end{equation*}\end{linenomath*}
			
			Then:
			\begin{linenomath*}\begin{equation} 
					(F^{4th})_{i+1}^n - (F^{4th})_{i}^n = \left(u(x_{i+\frac{1}{2}}) + \frac{u''(x_{i+\frac{1}{2}}) \Delta x^2}{8}\right) (Q_{i+1} - Q_i) + \frac{u'(x_{i+\frac{1}{2}}) \Delta x}{2} (Q_{i+1} + Q_i) + O(\Delta x^3) \label{eq:toexpand} 		 \end{equation}\end{linenomath*} 
			
			Using Taylor expansions of $\{ \phi_{i-\frac{3}{2}}, \phi_{i-\frac{1}{2}}, \dots, \phi_{i+\frac{5}{2}} \}$, immediately:
			\begin{linenomath*}\begin{equation*}
					Q_{i+1} - Q_{i} =  \phi'(x_{i+\frac{1}{2}}) \Delta x - \frac{\phi^{(5)}(x_{i+\frac{1}{2}}) \Delta x^5}{30} + O(\Delta x^7)
			\end{equation*}\end{linenomath*}
			\begin{linenomath*}\begin{equation*}
					Q_{i+1} + Q_{i} = 2\phi(x_{i+\frac{1}{2}}) + \frac{\phi''(x_{i+\frac{1}{2}}) \Delta x^2}{6}+ O(\Delta x^4)
			\end{equation*}\end{linenomath*}
			
			Replacing these values in equation \eqref{eq:toexpand}, and discarding error terms $\Delta x^k$ for $k \ge 3$, some computations give
			\begin{linenomath*}\begin{equation*}
					\frac{(F^{4th})_{i+1}^n - (F^{4th})_{i}^n}{\Delta x} = (u\phi)'(x_{i+\frac{1}{2}}) + O(\Delta x^2)
			\end{equation*}\end{linenomath*}
		\end{proof}
		\begin{remark} It is easy to check that all error terms of order $\Delta x^2$ have a derivative of $u$ as a factor, whereas the first error term with no derivatives of $u$ is $\frac{u \phi^{(5)}}{30} \Delta x^4$. Thus the design order is obtained for spatially constant advecting velocity.
		\end{remark}
		
		\subsection{Morinishi schemes}
		A second popular set of conservative approximations to the fluxes is described by Morinishi et al. \cite{Morinishi1998}. They can quite easily be described in full 3D generality, although some notation is required to write them in a convenient manner:
		
		\begin{linenomath*}\begin{equation}
				\label{eq:morinterp} \bar{\varphi}^{n,x}(x_i,y_j,z_k) := \frac{\varphi(x_i + \frac{n}{2}\Delta x, y_j, z_k) + \varphi(x_i - \frac{n}{2}\Delta x, y_j, z_k)}{2} \end{equation}\end{linenomath*}
		
		\begin{linenomath*}\begin{equation}\label{eq:morderiv} \frac{\delta_n \varphi}{\delta_n x}(x_i,y_j, z_k) := \frac{\varphi(x_i + \frac{n}{2} \Delta x, y_j,z_k) - \varphi(x_i - \frac{n}{2} \Delta x, y_j,z_k) }{n \Delta x} \end{equation}\end{linenomath*}
		
		The corresponding definitions for \eqref{eq:morinterp} (interpolation) and \eqref{eq:morderiv} (approximation to the derivative) in $y$ and $z$ are analogous. With these, the schemes can be written as follows.
		
		\subsubsection{4th order Morinishi scheme}
		
		\begin{linenomath*}\begin{align*} [\nabla \cdot (u \otimes u)]_i  & = \frac{9}{8} \sum_{j=1}^{3} \frac{\delta_1}{\delta_1 x_j} \left[ \left( \frac{9}{8} \bar{u}^{1,x_i}_j - \frac{1}{8} \bar{u}^{3,x_i}_j \right) \bar{u}^{1,x_j}_i \right] \\ & - \frac{1}{8}\sum_{j=1}^{3} \frac{\delta_3}{\delta_3 x_j}\left[ \left(  \frac{9}{8} \bar{u}^{1,x_i}_j - \frac{1}{8} \bar{u}^{3,x_i}_j \right) \bar{u}^{3,x_j}_i \right] + O(\Delta x^4) 
		\end{align*}\end{linenomath*}
		
		\subsubsection{6th order Morinishi scheme}
		
		\begin{linenomath*}\begin{align*} 
				[\nabla \cdot (u \otimes u)]_i & = \frac{150}{128} \sum_{j=1}^{3} \frac{\delta_1}{\delta_1 x_j}\left[ \left( \frac{150}{128} \bar{u}^{1,x_i}_j -\frac{25}{128}\bar{u}^{3,x_i}_j + \frac{3}{128} \bar{u}^{5,x_i}_j \right) \bar{u}^{1,x_j} _i \right]\\ 
				& - \frac{25}{128} \sum_{j=1}^{3} \frac{\delta_1}{\delta_1 x_j}\left[ \left( \frac{150}{128} \bar{u}^{1,x_i}_j -\frac{25}{128}\bar{u}^{3,x_i}_j + \frac{3}{128} \bar{u}^{5,x_i}_j \right) \bar{u}^{3,x_j} _i \right] \\
				& + \frac{3}{128} \sum_{j=1}^{3} \frac{\delta_1}{\delta_1 x_j}\left[ \left( \frac{150}{128} \bar{u}^{1,x_i}_j -\frac{25}{128}\bar{u}^{3,x_i}_j + \frac{3}{128} \bar{u}^{5,x_i}_j \right) \bar{u}^{5,x_j} _i \right] + O(\Delta x^6)
		\end{align*}\end{linenomath*}

	\end{appendices}

	\bibliography{paper.bib}
	\bibliographystyle{siam}
\end{document}